\documentclass[11pt,leqno]{amsart}

\usepackage[a4paper, right=20mm, left=20mm, top=20mm]{geometry}

\usepackage{amsthm}     
\usepackage[english]{babel} 
\usepackage[utf8]{inputenc} 
\usepackage{hyperref}
\usepackage{amsmath}
\usepackage{amssymb}
\usepackage{mathtools}
\usepackage{multicol}

\numberwithin{equation}{section}

\theoremstyle{plain}
\newtheorem{Theorem}[equation]{Theorem}
\newtheorem{thmalph}{Theorem}
\newtheorem{Lemma}[equation]{Lemma}
\newtheorem{Proposition}[equation]{Proposition}
\newtheorem{Corollary}[equation]{Corollary}

\newcommand{\EatDot}[1]{}

\theoremstyle{definition}
\newtheorem{Definition}[equation]{Definition}
\newtheorem{Remark}[equation]{Remark}

\newcommand{\C}{\mathbb{C}}
\newcommand{\R}{\mathbb{R}}

\newcommand{\Z}{\mathbb{Z}}

\newcommand{\res}{\mathrm{res}}
\newcommand{\id}{\mathrm{id}}

\newcommand{\Hom}{\mathrm{Hom}}
\newcommand{\End}{\mathrm{End}}

\newcommand{\hgt}{\mathrm{ht}}
\newcommand{\GL}{\mathrm{GL}}
\newcommand{\SL}{\mathrm{SL}}

\DeclareMathOperator{\rad}{rad}
\DeclareMathOperator{\soc}{soc}
\DeclareMathOperator{\head}{head}
\DeclareMathOperator{\ind}{ind}

\title[On complete reducibility of tensor products]{On complete reducibility of tensor products of simple modules over simple algebraic groups}
\author{Jonathan Gruber}
\date{\today}
\address{\'Ecole Polytechnique Federale de Lausanne, 1015 Lausanne, Switzerland}
\email{jonathan.gruber@epfl.ch}

\begin{document}

\renewcommand{\theequation}{\thesection.\arabic{equation}}
\newtagform{bold}[\textbf]{(}{)}

\begin{abstract}
	Let $G$ be a simply connected simple algebraic group over an algebraically closed field $k$ of characteristic $p>0$. The category of rational $G$-modules is not semisimple. We consider the question of when the tensor product of two simple $G$-modules $L(\lambda)$ and~$L(\mu)$ is completely reducible. Using some technical results about weakly maximal vectors (i.e. maximal vectors for the action of the Frobenius kernel $G_1$ of $G$) in tensor products, we obtain a reduction to the case where the highest weights $\lambda$ and~$\mu$ are $p$-restricted. In this case, we also prove that $L(\lambda)\otimes L(\mu)$ is completely reducible as a $G$-module if and only if $L(\lambda)\otimes L(\mu)$ is completely reducible as a $G_1$-module.
\end{abstract}

\maketitle

\section{Introduction}

Let $G$ be a simply connected simple algebraic group over an algebraically closed field $k$ of positive characteristic $p$. The simple $G$-modules are parametrized by the set $X^+$ of dominant weights of~$G$ (with respect to a fixed maximal torus and Borel subgroup) and for $\lambda\in X^+$, we write $L(\lambda)$ for the unique simple $G$-module of highest weight $\lambda$. One of the most powerful tools in examining the simple modules $L(\lambda)$ is Steinberg's tensor product theorem: Given $\lambda\in X^+$, there is a unique $p$-adic decomposition $\lambda=\lambda_0+p\lambda_1$, where $\lambda_0$ is a $p$-restricted weight and $\lambda_1\in X^+$. Then the simple module $L(\lambda)$ has a tensor product decomposition
\[ L(\lambda) \cong L(\lambda_0) \otimes L(\lambda_1)^{[1]} , \]
where $L(\lambda_1)^{[1]}$ denotes the Frobenius twist of the simple module $L(\lambda_1)$. Furthermore, the simple $G$-module $L(\lambda_0)$ remains simple upon restriction to the first Frobenius kernel $G_1$ of $G$ by a result of C.~W.~Curtis; see \cite{Curtis}. This allows one to reduce many questions about simple $G$-modules to questions about simple $G$-modules with $p$-restricted highest weight, or to questions about simple $G_1$-modules.

Given weights $\lambda,\mu\in X^+$ with $p$-adic decomposition $\lambda=\lambda_0+p\lambda_1$ and $\mu=\mu_0+p\mu_1$, respectively, the tensor product $L(\lambda)\otimes L(\mu)$ admits a decomposition
\[ L(\lambda)\otimes L(\mu) \cong L(\lambda_0) \otimes L(\lambda_1)^{[1]} \otimes L(\mu_0) \otimes L(\mu_1)^{[1]} \cong \big( L(\lambda_0) \otimes L(\mu_0) \big) \otimes \big( L(\lambda_1)\otimes L(\mu_1) \big)^{[1]} . \]
Thus, a lot of structural information about $L(\lambda)\otimes L(\mu)$ can be obtained by understanding the structure of $L(\lambda_0)\otimes L(\mu_0)$ and $L(\lambda_1)\otimes L(\mu_1)$. One of our main results, see Theorem \ref{thm:introC} below, is an illustration of this principle.

Our first main result is the following:
\begin{thmalph} \label{thm:introA}
	Let  $\lambda,\mu\in X^+$ be $p$-restricted. If $L(\lambda)\otimes L(\mu)$ is completely reducible then all composition factors of $L(\lambda)\otimes L(\mu)$ are $p$-restricted.
\end{thmalph}
Additionally, we obtain a theorem relating complete reducibility of $G$-modules and $G_1$-modules:
\begin{thmalph} \label{thm:introB}
	Let  $\lambda,\mu\in X^+$ be $p$-restricted. Then $L(\lambda)\otimes L(\mu)$ is completely reducible as a $G$-module if and only if $L(\lambda)\otimes L(\mu)$ is completely reducible as a $G_1$-module.
\end{thmalph}
Combining Theorems \ref{thm:introA} and \ref{thm:introB}, we obtain the following reduction theorem:
\begin{thmalph} \label{thm:introC}
	Let $\lambda,\mu\in X^+$ and write $\lambda=\lambda_0+\cdots+p^m\lambda_m$ and $\mu=\mu_0+\cdots+p^m\mu_m$ with $\lambda_0,\ldots,\lambda_m$ and $\mu_0,\ldots,\mu_m$ all $p$-restricted. Then $L(\lambda)\otimes L(\mu)$ is completely reducible if and only if $L(\lambda_i)\otimes L(\mu_i)$ is completely reducible for all $i$.
\end{thmalph}

The question of complete reducibility of tensor products of $G$-modules has previously been considered by J. Brundan and A. Kleshchev in \cite{BrundanKleshchevSocle} and \cite{BrundanKleshchevCR}, and by J.-P. Serre in \cite{Serre}. Some of their results will be recalled in Sections \ref{sec:someresults} and \ref{sec:Anp2} below.

We prove our results using some new techniques for weakly maximal vectors (that is, maximal vectors for the action of $G_1$) in tensor products of $G$-modules. More precisely, we give criteria under which weakly maximal vectors of non-$p$-restricted weights generate non-simple $G_1$-submodules (see Propositions \ref{prop:weakmaxvecnotG2} and \ref{prop:nonpreslevel0}) and we show how to construct explicitly a weakly maximal vector of weight~$\delta^\prime >\delta$, given a weakly maximal vector of weight~$\delta$ (see Propositions \ref{prop:weaklymaximalvector} and \ref{prop:nothighestweighthigherweakmaxvec}). In the proofs of Theorems \ref{thm:introA} and \ref{thm:introB}, we will use these results to construct weakly maximal vectors that generate non-simple $G_1$-submodules of $L(\lambda)\otimes L(\mu)$, thus showing that $L(\lambda)\otimes L(\mu)$ is not completely reducible as a $G_1$-module.

The paper is organized as follows: In Section \ref{sec:preliminaries}, we summarize the basic definitions and recall some important results. Section \ref{sec:weaklymaximalvectors} is concerned with the results on weakly maximal vectors in tensor products of $G$-modules that will be required to prove Theorems \ref{thm:introA} and \ref{thm:introB}. In section \ref{sec:someresults}, we cite results about complete reducibility from the literature and derive some consequences. The results we are using are due to H. H. Andersen, J. Brundan, A. Kleshchev, J.-P. Serre and I. Suprunenko. Some of the results in Section \ref{sec:weaklymaximalvectors} are only valid for groups of type different from $\mathrm{G}_2$ and for primes that are not too small with respect to the root system. Therefore, the proofs of Theorems \ref{thm:introA} and \ref{thm:introB} are split up over several sections. In Section \ref{sec:CRprestricted}, we will consider the case where $G$ is of type different from $\mathrm{G}_2$ and $p>2$ if $G$ is of type $\mathrm{B}_n$, $\mathrm{C}_n$ or $\mathrm{F}_4$. In Section \ref{sec:G2}, we give proofs of the theorems for~$G$ of type $\mathrm{G}_2$ when $p\neq 3$. Finally, if $G$ is of type $\mathrm{B}_n$, $\mathrm{C}_n$ or $\mathrm{F}_4$ and $p=2$ or $G$ is of type $\mathrm{G}_2$ and $p=3$ then the simple $G$-modules of $p$-restricted weight admit a refined tensor product decomposition corresponding to the decomposition of the root system of $G$ into short roots and long roots. We make use of this in Section \ref{sec:smallprimes} in order to prove Theorems \ref{thm:introA} and \ref{thm:introB} in the remaining cases. Our treatment of groups of type $\mathrm{B}_n$ in characteristic $p=2$ relies on a detailed study of tensor products of simple modules for the Levi subgroup of type $\mathrm{A}_{n-1}$. These results are given in Section \ref{sec:Anp2}, along with a complete classification of the pairs of $2$-restricted weights $\lambda$ and $\mu$ such that $L(\lambda)\otimes L(\mu)$ is completely reducible for $G$ of type~$\mathrm{A}_n$ when $p=2$. In the final Section \ref{sec:reductiontheorem}, we give the proof of Theorem \ref{thm:introC}.

\subsection*{Acknowledgements}
I thank Stephen Donkin for extremely fruitful discussions and guidance and for suggesting the proofs of Lemmas \ref{lem:DonkinNonCR} and \ref{lem:Donkin}. I would also like to thank my advisor Donna Testerman for her suggestions and careful reading of the manuscript. This work was supported by the Swiss National Science Foundation, grant number FNS 200020\textunderscore 175571.

\section{Preliminaries} \label{sec:preliminaries}

In this section, we give the basic definitions and cite some important results from the literature.

\subsection{Notation}

Our notational conventions are essentially the same as in \cite{Jantzen}, except that we write $\nabla(\lambda)$ for the induced module and $\Delta(\lambda)$ for the Weyl module of highest weight~$\lambda$. The following basic notations will be used throughout:

We fix $k$ to be an algebraically closed field of characteristic $p>0$ and $G$ to be a simply connected simple algebraic group scheme over $k$, defined and split over the finite field $\mathbb{F}_p$. The assumption of $G$ being simple and simply connected is for convenience and our main results generalize to connected reductive groups over $k$. Let $T$ be a split maximal torus in $G$ and denote by $X=X(T)$ the character group of~$T$. Let $\Phi\subseteq X$ be the root system of $G$ with respect to $T$, with a fixed choice of base $\Delta=\{\alpha_1,\ldots,\alpha_n\}$. Unless otherwise specified, we adopt the standard labeling of simple roots as given in \cite{Bourbaki}. We write $\Phi^+$ for the positive system defined by $\Delta$ and $\Phi^-=-\Phi^+$.
Let $W$ be the Weyl group of $\Phi$ and let $\langle\cdot\,,\cdot\rangle$ be a $W$-invariant inner product on the real space $X\otimes_\Z\R$, normalized so that $\langle\alpha,\alpha\rangle=2$ for all short roots~$\alpha\in\Phi$. The coroot of $\alpha\in\Phi$ is defined by $\alpha^\vee=2\alpha/\langle\alpha,\alpha\rangle$. Let
\[ X^+=\{ \lambda\in X \mid \langle \lambda , \alpha^\vee \rangle \geq 0 \text{ for all }\alpha\in\Delta \} \]
be the set of dominant weights, define
\[ X_1^\prime=\{ \lambda\in X \mid \langle \lambda,\alpha^\vee \rangle < p \text{ for all }\alpha\in\Delta \} \]
and set $X_1\coloneqq X_1^\prime \cap X^+$, the set of $p$-restricted (dominant) weights. Let $\omega_1,\ldots,\omega_n\in X^+$ be the fundamental dominant weights with respect to $\Delta$, that is $\langle\omega_i,\alpha_j^\vee\rangle=\delta_{ij}$,  and let $\rho=\omega_1+\cdots+\omega_n\in X^+$. There is a partial order on $X$ defined by $\lambda\geq\mu$ if and only if $\lambda-\mu$ is a non-negative integer linear combination of positive roots. Denote by $\alpha_0\in\Phi^+$ the highest root with respect to this partial order.

Denote by $F\colon G\to G$ a Frobenius endomorphism and let $G_1=\ker(F)$ be the first Frobenius kernel of $G$. For a rational $G$-module $M$, we denote by $M^{[1]}$ the Frobenius twist of $G$. If $M$ is finite-dimensional, we denote by $M^*$ the dual module of $M$ and by $M^\tau$ the contravariant dual of $M$ (see Section II.2.12 in \cite{Jantzen}). For $\mu\in X$, we denote by $M_\mu$ the $\mu$-weight space of $M$ and call its non-zero elements weight vectors of weight~$\mu$. We write~$B$ for the Borel subgroup of $G$ corresponding to $\Phi^-$ and define $\nabla(\lambda)=\ind_B^G(\lambda)$ for $\lambda\in X^+$. Finally, we write $\Delta(\lambda)=\nabla(\lambda)^\tau$ for the Weyl module of highest weight $\lambda\in X^+$ and $L(\lambda)=\soc_G \nabla(\lambda) =\head_G \Delta(\lambda)$ for the simple module of highest weight $\lambda$.

For $I\subseteq \{1,\ldots,n\}$, we denote by $L_I$ the derived subgroup of the Levi subgroup of $G$ corresponding to the simple roots $\{\alpha_i\mid i\in I\}$, a simply connected semisimple algebraic group. For $1\leq i\leq j\leq n$, we write $[i,j]$ for the set $\{i,i+1,\ldots, j\}$.

\subsection{The hyperalgebra and its infinitesimal subalgebra}

Instead of working with the group schemes $G$ and $G_1$ directly, we will be using the hyperalgebra of~$G$ and its infinitesimal subalgebra, which will enable us to carry out explicit constructions of weakly maximal vectors in tensor products in Section \ref{sec:weaklymaximalvectors}.
Let $\mathfrak{g}$ be the complex simple Lie algebra with root system $\Phi$, let
	\[\{ X_\alpha, H_i \mid \alpha\in\Phi,i=1,\ldots,n \}\]
be a Chevalley basis of $\mathfrak{g}$ and denote by $U(\mathfrak{g})$ the universal enveloping algebra of $\mathfrak{g}$.

\begin{Definition}
\begin{enumerate}
	\item The \emph{Kostant $\Z$-form} $U_\Z(\mathfrak{g})$ of $U(\mathfrak{g})$ is the $\Z$-subalgebra of $U(\mathfrak{g})$ generated by the \emph{divided powers} $X_{\alpha,r}=X_\alpha^r/r!$ and $\binom{H_i}{m}=\frac{H_i(H_i-1)\cdots(H_i-m+1)}{m!}$ for $\alpha\in\Phi$, $i=1,\ldots,n$ and $r,m\geq 0$.
	\item The \emph{hyperalgebra} of $G$ is the $k$-algebra $U_k(\mathfrak{g})=U_\Z(\mathfrak{g})\otimes_\Z k$.
\end{enumerate}
\end{Definition}

In the following, we will write $X_{\alpha,r}$ and $\binom{H_i}{m}$ instead of $X_{\alpha,r}\otimes 1_k$ and $\binom{H_i}{m}\otimes 1_k$ for the images of the divided powers in $U_k(\mathfrak{g})$, and we abbreviate $X_{\alpha,1}$ by $X_\alpha$.
\medskip

Recall that $U(\mathfrak{g})$ is a Hopf $\C$-algebra with comultiplication, counit and antipode given by 
\[x\xmapsto{\Delta}x\otimes 1+ 1\otimes x,\qquad x\xmapsto{\varepsilon}0\qquad\text{and}\qquad x\xmapsto{\sigma}-x ,\]
respectively, for elements $x\in\mathfrak{g}$. These maps restrict to $U_\Z(\mathfrak{g})$ and therefore make~$U_k(\mathfrak{g})$ into a Hopf $k$-algebra whose structure maps we also denote by $\Delta$, $\varepsilon$ and $\sigma$.

	As shown in Section II.1.12 in \cite{Jantzen}, $U_k(\mathfrak{g})$ is isomorphic to the \emph{distribution algebra} of $G$ (recall that we assume $G$ to be simple and simply connected), so every rational $G$-module is in a natural way a locally finite $U_k(\mathfrak{g})$-module. Every locally finite $U_k(\mathfrak{g})$-module can be equipped with the structure of a rational $G$-module and this induces an equivalence of categories between $\{\text{rational }G\text{-modules}\}$ and $\{ \text{locally finite }U_k(\mathfrak{g})\text{-modules} \}$. For a rational $G$-module $V$ and $\lambda\in X$, we have $X_{\alpha,r}\cdot V_\lambda\subseteq V_{\lambda+r\alpha}$ for all $\alpha\in\Phi$ and $r\geq 0$. See Sections II.1.19 and II.1.20 in \cite{Jantzen} for more details.

\begin{Remark} \label{rem:hyperalgebratensorproduct}
	For rational $G$-modules $V$ and $W$, the tensor product $V\otimes W$ becomes a rational $G$-module via $g\cdot (v\otimes w)=(gv)\otimes (gw)$ for $v\in V$, $w\in W$ and $g\in G$. The corresponding $U_k(\mathfrak{g})$-module structure on $V\otimes W$ is obtained by pulling back the natural action of $U_k(\mathfrak{g})\otimes U_k(\mathfrak{g})$ along the comultiplication map $\Delta$. In particular, we have
	\[ X_\alpha \cdot (v\otimes w) = \Delta(X_\alpha)\cdot (v\otimes w) = (X_\alpha\cdot v)\otimes w + v\otimes (X_\alpha\cdot w) \]
	for all $v\in V$, $w\in W$ and $\alpha\in\Phi$.
\end{Remark}

\begin{Definition}
	The first \emph{infinitesimal subalgebra} $u_k(\mathfrak{g})$ of $U_k(\mathfrak{g})$ is the subalgebra generated by $X_\alpha$ and~$\binom{H_i}{1}$ for $\alpha\in\Phi$ and $i=1,\ldots,n$.
\end{Definition}

The infinitesimal subalgebra $u_k(\mathfrak{g})$ is isomorphic to the restricted universal enveloping algebra of the Lie algebra of $G$. Every $G_1$-module is in a natural way a $u_k(\mathfrak{g})$-module and every $u_k(\mathfrak{g})$-module is in a natural way a $G_1$-module, which yields an equivalence of categories between $\{ G_1\text{-modules} \}$ and $\{ u_k(\mathfrak{g})\text{-modules} \}$. See Sections I.8.4, I.8.6, I.9.6 and II.3.3 in~\cite{Jantzen} for more details.


\section{Weakly maximal vectors in tensor products} \label{sec:weaklymaximalvectors}

Let us begin with the definition of a weakly maximal vector.

\begin{Definition}
	Let $V$ be a $G$-module. A \emph{weakly maximal vector} is a non-zero weight vector $v\in V$ such that $X_\alpha \cdot v=0$ for all $\alpha\in\Phi^+$.
\end{Definition}

Weakly maximal vectors have previously been considered by J. Brundan and A. Kleshchev in the proof of Theorem 3.3 in \cite{BrundanKleshchevSocle}, where they were called weakly primitive vectors. Our Lemma \ref{lem:weaklymaximaltensor} below was inspired by a computation in the aforementioned proof.

In this section, we prove some results about weakly maximal vectors in tensor products of $G$-modules that will be crucial for the proofs of the main results in Section \ref{sec:CRprestricted}. We first describe a way to construct explicitly from a weakly maximal vector of weight $\delta$ another weakly maximal vector of weight $\delta+\beta$ for some $\beta\in\Phi^+$, under some mild assumptions on $\delta$. Then we give criteria under which a weakly maximal vector of non-$p$-restricted weight generates a non-simple $G_1$-submodule.

\begin{Remark} \label{rem:WeakMaxVecGeneratesSimple}
	Suppose that $V$ is a rational $G$-module and that $v\in V$ is a weakly maximal vector of weight $\delta\in X$. If $\delta=\delta_0+p\delta_1$ with $\delta_0\in X_1$ and $\delta_1\in X$ (not necessarily dominant) then $v$ generates a $G_1$-submodule $U=u_k(\mathfrak{g})\cdot v$ of $V$ with $\head_{G_1}(U)\cong L(\delta_0)\vert_{G_1}$. If $V$ is completely reducible as a $G_1$-module then it follows that every weakly maximal vector $v\in V$ generates a simple $G_1$-submodule. Producing weakly maximal vectors that generate non-simple $G_1$-submodules will be our main tool for establishing non-complete reducibility, see for instance Propositions \ref{prop:nonpresnonsimpleG1} and \ref{prop:nonpreslevel0} below. 
\end{Remark}

\begin{Proposition} \label{prop:weaklymaximalvector}
	Let $V$ and $W$ be rational $G$-modules and let $v$ be a weakly maximal vector in~$V\otimes W$. Suppose that there exists $\alpha\in\Phi^+$ such that $\left( 1 \otimes X_\alpha \right)\cdot v\neq 0$ and let $\beta\in\Phi^+$ be maximal with the property that $\left( 1\otimes X_\beta \right)\cdot v\neq 0$. Then $\left( 1\otimes X_\beta \right)\cdot v$ is a weakly maximal vector.
\end{Proposition}
\begin{proof}
	For $\gamma\in\Phi^+$, there exists $c\in k$ such that $X_\gamma X_\beta=X_\beta X_\gamma + c \cdot X_{\beta+\gamma}$, where we use the convention that $X_{\beta+\gamma}=0$ if $\beta+\gamma\notin\Phi$. It follows that
	\begin{align*}
	\Delta\left(X_\gamma\right)\cdot \left( 1 \otimes X_\beta \right) & = \left( X_\gamma\otimes 1 + 1\otimes X_\gamma \right) \cdot \left( 1 \otimes X_\beta \right) \\
	& = X_\gamma\otimes X_\beta + 1\otimes X_\gamma X_\beta \\
	& = X_\gamma\otimes X_\beta + 1\otimes X_\beta X_\gamma + c \cdot 1\otimes X_{\beta+\gamma} \\
	& = (1\otimes X_\beta) \cdot (X_\gamma\otimes 1 + 1\otimes X_\gamma) + c\cdot 1\otimes X_{\gamma+\beta} \\
	& = (1\otimes X_\beta) \cdot \Delta(X_\gamma) + c\cdot 1\otimes X_{\gamma+\beta} .
	\end{align*}
	Now $\left( 1 \otimes X_{\beta+\gamma} \right)\cdot v=0$ by maximality of $\beta$, and $\Delta(X_\gamma)\cdot v=0$ as $v$ is a weakly maximal vector. We conclude that $\Delta(X_\gamma)\cdot \left( 1 \otimes X_\beta \right)\cdot v=0$, so $\left( 1\otimes X_\beta \right)\cdot v$ is a weakly maximal vector.
\end{proof}

\begin{Lemma} \label{lem:weaklymaximaltensor}
	Let $\lambda,\mu\in X_1$ and let $v\in L(\lambda)\otimes L(\mu)$ be a weakly maximal vector of weight $\delta$. Suppose that
	\[ v=\sum_{\gamma+\vartheta=\delta} v_{(\gamma,\vartheta)}, \]
	with $v_{(\gamma,\vartheta)}\in L(\lambda)_\gamma\otimes L(\mu)_\vartheta$. Then $v_{(\lambda,\delta-\lambda)}\neq 0$ and $v_{(\delta-\mu,\mu)}\neq 0$.
\end{Lemma}
\begin{proof}
	By symmetry, it suffices to show that $v_{(\lambda,\delta-\lambda)}\neq 0$. Let $\gamma_0$ be maximal with $v_{(\gamma_0,\delta-\gamma_0)}\neq 0$ and write
	\[ v_{(\gamma_0,\delta-\gamma_0)} = v_1\otimes w_1 + \cdots + v_s\otimes w_s \]
	with $v_1,\ldots,v_s \in L(\lambda)_{\gamma_0}$ all non-zero and $w_1,\ldots,w_s \in L(\mu)_{\delta-\gamma_0}$ linearly independent. We will show that $\gamma_0=\lambda$.
	
	For $\nu,\eta\in X$, denote by $p_{(\nu,\eta)}$ the projection from $L(\lambda)\otimes L(\mu)=\bigoplus_{\gamma,\vartheta} L(\lambda)_\gamma\otimes L(\mu)_\vartheta$ onto the summand $L(\lambda)_\nu\otimes L(\mu)_\eta$. For $\alpha\in\Phi^+$, we have $X_\alpha\cdot v=0$ and therefore
	\[ 0 = p_{(\gamma_0+\alpha,\delta-\gamma_0)}\left( X_\alpha\cdot v \right)=\sum_{\gamma+\vartheta=\delta} p_{(\gamma_0+\alpha,\delta-\gamma_0)}\left( X_\alpha\cdot v_{(\gamma,\vartheta)} \right) . \]
	Now if	$p_{(\gamma_0+\alpha,\delta-\gamma_0)}\left( X_\alpha\cdot v_{(\gamma,\delta-\gamma)} \right)\neq 0$ for some $\gamma\in X$ then $\gamma_0+\alpha\in\{\gamma,\gamma+\alpha\}$ as
	\[ X_{\alpha}\cdot v_{(\gamma,\delta-\gamma)}\in L(\lambda)_\gamma\otimes L(\mu)_{\delta-\gamma+\alpha} \oplus L(\lambda)_{\gamma+\alpha} \otimes L(\mu)_{\delta-\gamma} \]
	by Remark \ref{rem:hyperalgebratensorproduct}, so $\gamma\geq\gamma_0$ and $\gamma=\gamma_0$ by maximality of $\gamma_0$.
	We conclude that
	\[ 0 = p_{(\gamma_0+\alpha,\delta-\gamma_0)}\left( X_\alpha\cdot v \right) = p_{(\gamma_0+\alpha,\delta-\gamma_0)}\left( X_\alpha\cdot v_{(\gamma_0,\delta-\gamma_0)} \right) = (X_\alpha\cdot v_1)\otimes w_1 + \cdots + (X_\alpha\cdot v_s)\otimes w_s \]
	and linear independence of $w_1,\ldots,w_s$ implies that $X_\alpha\cdot v_1=0$. Hence $v_1$ is a weakly maximal vector in $L(\lambda)$ of weight $\gamma_0$ and we conclude that $\gamma_0=\lambda$.
\end{proof}

\begin{Proposition} \label{prop:nothighestweighthigherweakmaxvec}
	Let $\lambda,\mu\in X_1$ and let $v\in L(\lambda)\otimes L(\mu)$ be a weakly maximal vector of weight $\delta\neq\lambda+\mu$. Then there exists $\beta\in\Phi^+$ such that $(1\otimes X_\beta)\cdot v$ is a weakly maximal vector in $L(\lambda)\otimes L(\mu)$.
\end{Proposition}
\begin{proof}
	Write $v=\sum v_{(\gamma,\delta-\gamma)}$ with $v_{(\gamma,\delta-\gamma)}\in L(\lambda)_\gamma\otimes L(\mu)_{\delta-\gamma}$. Then $0\neq v_{(\lambda,\delta-\lambda)}$ by Lemma \ref{lem:weaklymaximaltensor} and therefore $v_{(\lambda,\delta-\lambda)}=v^+\otimes w$ for some non-zero $v^+\in L(\lambda)_\lambda$ and $w\in L(\mu)_{\delta-\lambda}$. As $\delta-\lambda\neq\mu$, $w$ is not a weakly maximal vector in $L(\mu)$ and there exists $\gamma\in\Phi^+$ such that $X_\gamma\cdot w\neq 0$. It follows that $\left(1\otimes X_\gamma\right)\cdot v\neq 0$ and we choose $\beta \in\Phi^+$ maximal with the property that $\left(1\otimes X_\beta\right)\cdot v\neq 0$. Then $\left(1\otimes X_\beta\right)\cdot v$ is a weakly maximal vector by Proposition \ref{prop:weaklymaximalvector}.
\end{proof}

Now we establish some criteria under which a weakly maximal vector of non-$p$-restricted weight generates a non-simple $G_1$-submodule. Most of these criteria are based on the following lemma:

\begin{Lemma} \label{lem:weakmaxvecnonsimpleG1}
	Let $V$ be a rational $G$-module and $v\in V$ a weakly maximal vector of weight $\delta$. Suppose that there exists $\alpha\in\Delta$ such that $p\nmid \langle\delta,\alpha^\vee\rangle+1$ and write $\langle\delta,\alpha^\vee\rangle+1=p\cdot q+ r$ with $0<r<p$ and~$q\in\Z$. If $v$ generates a simple $G_1$-submodule of $V$ then $X_{-\alpha,r}\cdot v=0$.
\end{Lemma}
\begin{proof}
	Write $\delta=\delta_0+p\delta_1$ with $\delta_0\in X_1$ and $\delta_1\in X$. Assume that $v$ generates a simple $G_1$-submodule of $V$, so that $u_k(\mathfrak{g})\cdot v\cong  L(\delta_0)\vert_{G_1}$ by Remark \ref{rem:WeakMaxVecGeneratesSimple}.
	We have
	\[ p\cdot q + r = \langle\delta,\alpha^\vee\rangle + 1 = \langle\delta_0+p\delta_1,\alpha^\vee\rangle + 1 = p\cdot \langle\delta_1,\alpha^\vee\rangle + \langle\delta_0,\alpha^\vee\rangle + 1 \]
	and it follows that $r=\langle\delta_0,\alpha^\vee\rangle+1$ as $\delta_0\in X_1$ and $p\nmid \langle\delta,\alpha^\vee\rangle+1$. Hence $\delta_0-r\alpha=s_\alpha(\delta_0)-\alpha$ and $s_\alpha(\delta_0-r\alpha)=\delta_0+\alpha$, where $s_\alpha\in W$ denotes the reflection along $\alpha$. We conclude that $\delta_0-r\alpha$ is not a weight of $L(\delta_0)$, so $X_{-\alpha,r}\cdot v=0$. 
\end{proof}

\begin{Corollary} \label{cor:highestweightvectorprestricted}
	Let $\lambda,\mu\in X_1$ and let $v$ and $w$ be maximal vectors in $L(\lambda)$ and $L(\mu)$, respectively. If the maximal vector $v\otimes w$ generates a simple $G_1$-submodule of $L(\lambda)\otimes L(\mu)$ then $\lambda+\mu\in X_1$.
\end{Corollary}
\begin{proof}
	Suppose for a contradiction that $\langle\lambda+\mu,\alpha^\vee\rangle\geq p$ for some $\alpha\in\Delta$. We have $X_{-\alpha,r}v\neq 0$ for $r=0,\ldots,\langle\lambda,\alpha^\vee\rangle$ and $X_{-\alpha,s}w\neq 0$ for $s=0,\ldots,\langle\mu,\alpha^\vee\rangle$ and it follows that
	\[ X_{-\alpha,t}\cdot (v\otimes w) = \sum_{i=0}^t (X_{-\alpha,i}\cdot v) \otimes (X_{-\alpha,t-i}\cdot w) \neq 0 \]
	for $t=0,\ldots,\langle\lambda+\mu,\alpha^\vee\rangle$. If $\langle\lambda+\mu,\alpha^\vee\rangle=q\cdot p+r$ with $q\in\Z$ and $0\leq r<p$ then clearly $r<\langle\lambda+\mu,\alpha^\vee\rangle$. Moreover, we have $p<\langle\lambda+\mu,\alpha^\vee\rangle+1<2p$ as $\langle\lambda,\alpha^\vee\rangle< p$ and $\langle\mu,\alpha^\vee\rangle< p$, so $p\nmid \langle\lambda+\mu,\alpha^\vee\rangle+1$. Now Lemma~\ref{lem:weakmaxvecnonsimpleG1} yields that $v\otimes w$ generates a non-simple $G_1$-submodule of $L(\lambda)\otimes L(\mu)$, a contradiction.
\end{proof}

In order to prove the next proposition, we first need an easy lemma about root systems.

\begin{Lemma} \label{lem:rootnegativecoeff}
	Let $\Phi$ be a root system of type different from $\mathrm{G}_2$ and let $\alpha,\beta\in\Phi^+$ such that $\langle\beta,\alpha^\vee\rangle<0$. Then $\langle\beta,\alpha^\vee\rangle\in\{-1,-2\}$ and $\beta-\alpha\notin\Phi$.
\end{Lemma}
\begin{proof}
	Consider the subsystem $\Phi^\prime\coloneqq(\Z\alpha+\Z\beta)\cap\Phi$ of $\Phi$. As $\langle\beta,\alpha^\vee\rangle<0$, we have $\alpha+\beta\in\Phi^\prime$. Suppose for a contradiction that $\beta-\alpha\in\Phi^\prime$. Then $\Phi^\prime$ contains a root string of length at least $3$, so $\Phi^\prime$ is of type $\mathrm{B}_2$ or $\mathrm{G}_2$, hence of type $\mathrm{B}_2$ as $\Phi$ is not of type $\mathrm{G}_2$ and therefore does not have a subsystem of type $\mathrm{G}_2$. Then $\{\beta-\alpha,\beta,\beta+\alpha\}$ is the $\alpha$-string through $\beta-\alpha$, so $\langle\beta-\alpha,\alpha^\vee\rangle=-2$ and $\langle\beta,\alpha^\vee\rangle=0$, a contradiction.
\end{proof}

The next result will be very useful in view of Proposition \ref{prop:nothighestweighthigherweakmaxvec}. In the proposition, we constructed from a weakly maximal vector $v\in L(\lambda)\otimes L(\mu)$ of weight $\delta\in X$ a weakly maximal vector $(1\otimes X_\beta)\cdot v$ for some $\beta\in\Phi^+$. Now we establish conditions on $\delta$ and $\beta$ under which $v$ generates a non-simple $G_1$-submodule.

\begin{Proposition} \label{prop:weakmaxvecnotG2}
	Assume that $\Phi$ is of type different from $\mathrm{G}_2$. Let $\lambda,\mu\in X_1$ and suppose that there is a weakly maximal vector $v\in L(\lambda)\otimes L(\mu)$ of weight $\delta\in X$ such that $\langle\delta,\alpha^\vee\rangle\geq p$ for some $\alpha\in\Delta$. Assume furthermore that there exists $\beta\in\Phi^+$ with $-\langle\beta,\alpha^\vee\rangle<p$ such that $w\coloneqq (1\otimes X_\beta)\cdot v$ is a weakly maximal vector in $L(\lambda)\otimes L(\mu)$ and $\langle\delta+\beta,\alpha^\vee\rangle<p$. Then $v$ generates a non-simple $G_1$-submodule of $L(\lambda)\otimes L(\mu)$.
\end{Proposition}
\begin{proof}
	We have $\langle\beta,\alpha^\vee\rangle = \langle\delta+\beta,\alpha^\vee\rangle - \langle\delta,\alpha^\vee\rangle <0$ and therefore by Lemma \ref{lem:rootnegativecoeff}, $\beta-\alpha\notin\Phi$ and $\langle\beta,\alpha^\vee\rangle\in\{-1,-2\}$. Hence $X_{-\alpha,s}$ and $X_\beta$ commute for all $s\geq 0$. Thus we have
	\begin{align*}
	X_{-\alpha,s}\cdot w &= \left(\sum_{i=0}^s X_{-\alpha,i}\otimes X_{-\alpha,s-i}\right)\cdot(1\otimes X_\beta)\cdot v \\
	&= (1\otimes X_\beta) \cdot \left(\sum_{i=0}^s X_{-\alpha,i}\otimes X_{-\alpha,s-i}\right)\cdot v \\ 
	&= (1\otimes X_\beta)\cdot X_{-\alpha,s}\cdot v,
\end{align*}
	so $X_{-\alpha,s}\cdot v\neq 0$ whenever $X_{-\alpha,s}\cdot w\neq 0$.
	
	 Furthermore, as $\langle\beta,\alpha^\vee\rangle\in\{-1,-2\}$ we have $\langle\delta+\beta,\alpha^\vee\rangle\in\{p-2,p-1\}$ and $\langle\delta,\alpha^\vee\rangle\in\{p,p+1\}$. Writing $\langle\delta,\alpha^\vee\rangle+1=p+r$ with $r\in\{1,2\}$, we have
	 \[r = \langle \delta,\alpha^\vee \rangle + 1 - p \leq \langle\delta+\beta,\alpha^\vee\rangle \]
	 as $-\langle\beta,\alpha^\vee\rangle<p$. By considering the restriction of $U_k(\mathfrak{g})\cdot w$ to the Levi subgroup $L_\alpha$ with root system~$\{\alpha,-\alpha\}$, we find that $X_{-\alpha,r}\cdot w\neq 0$, so $X_{-\alpha,r}\cdot v\neq 0$ and the claim follows from Lemma~\ref{lem:weakmaxvecnonsimpleG1}.
\end{proof}

\begin{Proposition} \label{prop:nonpresnonsimpleG1}
	Assume that $\Phi$ is of type different from $\mathrm{G}_2$. Let $\lambda,\mu\in X_1$ and suppose that there is a weakly maximal vector $v\in L(\lambda)\otimes L(\mu)$ of weight $\delta\in X$ such that $\langle\delta,\alpha^\vee\rangle\geq p$ for some $\alpha\in\Delta$ such that $-\langle\beta,\alpha^\vee\rangle<p$ for all $\beta\in\Phi^+$. Then $L(\lambda)\otimes L(\mu)$ has a weakly maximal vector that generates a non-simple $G_1$-submodule.
\end{Proposition}
\begin{proof}
	Let $Y=\{ \chi\in X \mid \langle \chi,\alpha^\vee \rangle\geq p\}\subseteq X$ and let $\delta^\prime\in Y$ be maximal with the property that $L(\lambda)\otimes L(\mu)$ has a weakly maximal vector $w$ of weight $\delta^\prime$. If $\delta^\prime=\lambda+\mu$ then $w$ generates a non-simple $G_1$-submodule by Corollary~\ref{cor:highestweightvectorprestricted}. If $\delta^\prime\neq\lambda+\mu$ then there exists $\beta\in\Phi^+$ such that $(1\otimes X_\beta)\cdot w$ is a weakly maximal vector in $L(\lambda)\otimes L(\mu)$ by Proposition \ref{prop:nothighestweighthigherweakmaxvec}. Then $\langle\delta^\prime+\beta,\alpha^\vee\rangle<p$ by maximality of $\delta^\prime$ and $w$ generates a non-simple $G_1$-submodule by Proposition \ref{prop:weakmaxvecnotG2}.
\end{proof}

Recall that $X_1^\prime=\{ \lambda\in X \mid \langle \lambda,\alpha^\vee \rangle < p \text{ for all }\alpha\in\Delta \}$. The following result is crucial for our proofs of Theorems \ref{thm:introA} and \ref{thm:introB} in Section \ref{sec:CRprestricted}. 

\begin{Corollary} \label{cor:nonpresnonsimpleG1}
	Assume that $\Phi$ is of type different from $\mathrm{G}_2$ and that $p>2$ if $\Phi$ is not simply laced. Let $\lambda,\mu\in X_1$ and suppose that $L(\lambda)\otimes L(\mu)$ has a weakly maximal vector of weight $\delta\in X\setminus X_1^\prime$. Then $L(\lambda)\otimes L(\mu)$ is not completely reducible as a $G_1$-module.
\end{Corollary}
\begin{proof}
	The assumption implies that $-\langle\beta,\alpha^\vee\rangle<p$ for all $\alpha\in\Delta$ and $\beta\in\Phi^+$ and the claim is immediate from Remark \ref{rem:WeakMaxVecGeneratesSimple} and Proposition \ref{prop:nonpresnonsimpleG1}.
\end{proof}

The following result is due to I. Suprunenko, see page 20 in \cite{Suprunenko}. For the convenience of the reader, we include a proof here. For $\gamma=a_1\alpha_1+\cdots+a_n\alpha_n\in\sum_{\alpha\in\Delta}\Z\alpha$, write $\hgt_{\alpha_i}(\gamma)\coloneqq a_i$.

\begin{Proposition}[Suprunenko] \label{prop:Suprunenko}
	Suppose that $p>2$ if $\Phi$ is of type $\mathrm{B}_n$, $\mathrm{C}_n$ or $\mathrm{F}_4$ and $p>3$ if $\Phi$ is of type $\mathrm{G}_2$. Let $\lambda\in X_1$ and let $0\neq v\in L(\lambda)$ be a weight vector of weight $\delta\in X$ such that $\hgt_\alpha(\lambda-\delta)=0$ for some $\alpha\in\Delta$. Then $X_{-\alpha,r}\cdot v\neq 0$ for $r=0,\ldots,\langle\delta,\alpha^\vee\rangle$.
\end{Proposition}
\begin{proof}
	The claim is true if $\delta=\lambda$ as can easily be seen by considering the restriction of $L(\lambda)$ to the Levi subgroup $L_\alpha$ with root system $\{\alpha,-\alpha\}$. Now assume for a contradiction that $\delta<\lambda$ is maximal with the property that $\hgt_\alpha(\lambda-\delta)=0$ and $X_{-\alpha,t}\cdot v=0$ for some weight vector $v\in L(\lambda)$ of weight $\delta$ and some $t\in\{0,\ldots,\langle\delta,\alpha^\vee\rangle\}$. As $\lambda$ is $p$-restricted and $\delta\neq\lambda$, $v$ is not a weakly maximal vector in $L(\lambda)$ and there exists $\gamma\in\Phi^+$ such that $X_\gamma\cdot v\neq 0$. By the assumptions on $p$ and properties of Chevalley bases, it follows that there exists $\beta\in\Delta$ with $X_\beta\cdot v\neq 0$, in particular $\beta\neq\alpha$. Suppose first that $t\leq -\langle\beta,\alpha^\vee\rangle$. Then $\beta+s\alpha\in\Phi$ for $0\leq s\leq t$, and for $s<t$ there exists $c_s\in k^\times$ such that $c_s\cdot X_{\beta+s\alpha}=[X_{\beta+(s+1)\alpha},X_{-\alpha}]$, by the assumptions on $p$ and $\Phi$. Then an easy induction using the assumption that $\hgt_\alpha(\lambda-\delta)=0$ shows that $X_{\beta+s\alpha}X_{-\alpha}^s\cdot v =c_0 c_1\cdots c_{s-1} X_\beta\cdot v\neq 0$ for $0\leq s\leq t$. We conclude that $X_{-\alpha}^t\cdot v\neq 0$ and therefore $X_{-\alpha,t}\cdot v\neq 0$, a contradiction.
	
	Now suppose that $t>-\langle\beta,\alpha^\vee\rangle$. We have $\delta+\beta>\delta$ and $\hgt_\alpha(\lambda-\delta-\beta)=0$. By maximality of $\delta$, it follows that $X_{-\alpha,r}X_{\beta}\cdot v\neq 0$ for $r=0,\ldots,\langle\delta+\beta,\alpha^\vee\rangle$. Furthermore, $X_\beta$ and $X_{-\alpha,r}$ commute for all $r> 0$ as $\beta-\alpha\notin\Phi$ and we conclude that $X_{-\alpha,r}\cdot v\neq 0$ for $r=0,\ldots,\langle\delta+\beta,\alpha^\vee\rangle$. Now
	\[U\coloneqq \mathrm{span}_k\{X_{-\alpha,r}\cdot v\mid r\geq 0\}\]
	is an $L_\alpha$-submodule of $L(\lambda)$, so $U$ is invariant under the reflection $s_\alpha$ which sends the weight space $U_{\delta-r\alpha}$ to $U_{\delta-\langle\delta,\alpha^\vee\rangle\alpha+r\alpha}$. It follows that $X_{-\alpha,\langle\delta,\alpha^\vee\rangle-r}\cdot v\neq 0$ for $r=0,\ldots,\langle\delta+\beta,\alpha^\vee\rangle$, in particular $X_{-\alpha,t}\cdot v\neq 0$ as $t>-\langle\beta,\alpha^\vee\rangle$, a contradiction.
\end{proof}

\begin{Proposition} \label{prop:nonpreslevel0}
	Suppose that $p>2$ if $\Phi$ is of type $\mathrm{B}_n$, $\mathrm{C}_n$ or $\mathrm{F}_4$ and $p>3$ if $\Phi$ is of type $\mathrm{G}_2$.
	Let~$\lambda,\mu\in X_1$ and suppose that $L(\lambda)\otimes L(\mu)$ has a weakly maximal vector $v$ of non-$p$-restricted weight~$\delta$ and let $\alpha\in\Delta$ such that $\langle\delta,\alpha^\vee\rangle\geq p$. Assume additionally that $\hgt_\alpha(\lambda+\mu-\delta)=0$ and $p\nmid \langle \delta , \alpha^\vee \rangle + 1$.
	Then $v$ generates a non-simple $G_1$ submodule of $L(\lambda)\otimes L(\mu)$.
\end{Proposition}
\begin{proof}
	We may write $v=\sum_\gamma v_{(\gamma,\delta-\gamma)}$ with $v_{(\gamma,\delta-\gamma)}\in L(\lambda)_\gamma\otimes L(\mu)_{\delta-\gamma}$ and $v_{(\lambda,\delta-\lambda)}\neq 0$ by Lemma~\ref{lem:weaklymaximaltensor}, so that $v_{(\lambda,\delta-\lambda)}=v^+\otimes w_{\delta-\lambda}$ for some $0\neq v^+\in L(\lambda)_\lambda$ and $0\neq w_{\delta-\lambda}\in L(\mu)_{\delta-\lambda}$.
	
	Note that by Proposition \ref{prop:Suprunenko}, we have $X_{-\alpha,s}\cdot w_{\delta-\lambda}\neq 0$ for all $s$ with $0\leq s\leq \langle\delta-\lambda,\alpha^\vee \rangle$. Now writing $\langle\delta,\alpha^\vee \rangle+1=q\cdot p+r$ with $0\leq r<p$, we have $q\geq 1$ as $\langle\delta,\alpha^\vee \rangle\geq p$ and therefore
	\[r\leq \langle\delta,\alpha^\vee \rangle-(p-1) \leq \langle\delta-\lambda,\alpha^\vee \rangle,\]
	as $\lambda$ is $p$-restricted. Hence $X_{-\alpha,r}\cdot w_{\delta-\lambda}\neq 0$ and therefore $X_{-\alpha,r} v\neq 0$. Indeed, writing $p_{(\gamma,\vartheta)}$ for the linear projection from $L(\lambda)\otimes L(\mu)$ onto $L(\lambda)_\gamma\otimes L(\mu)_\vartheta$ as in the proof of Lemma \ref{lem:weaklymaximaltensor}, we have
	\[p_{(\lambda,\delta-\lambda-r\alpha)}\left(X_{-\alpha,r}\cdot v_{(\gamma,\delta-\gamma)}\right)=0\]
	for all $\gamma<\lambda$ and it follows that
	\[ p_{(\lambda,\delta-\lambda-r\alpha)}\left(X_{-\alpha,r} \cdot v\right) = p_{(\lambda,\delta-\lambda-r\alpha)}\left(X_{-\alpha,r} \cdot v_{(\lambda,\delta-\lambda)}\right) = v^+\otimes ( X_{-\alpha,r}\cdot w_{\delta-\lambda} ) \neq 0 , \]
	so $X_{-\alpha,r}\cdot v\neq 0$.
	Now the claim follows from Lemma \ref{lem:weakmaxvecnonsimpleG1}.
\end{proof}

\section{Some results about complete reducibility} \label{sec:someresults}

In this section, we cite some results from the literature and prove some important preliminary results.
We will need the following theorem of J. Brundan and A. Kleshchev; see Theorem B in~\cite{BrundanKleshchevSocle}.

\begin{Theorem}[Brundan-Kleshchev] \label{thm:BrundanKleshchevsocle}
	Let $\lambda,\mu\in X_1$ with $\langle\lambda,\alpha_0^\vee\rangle<p$. Then the socle of $\nabla(\lambda) \otimes \nabla(\mu)$ is $p$-restricted. In particular, the socle of $L(\lambda)\otimes L(\mu)$ is $p$-restricted.
\end{Theorem}

Recall that a finite dimensional $G$-module $M$ is said to have a \emph{good filtration} if there exists a sequence of $G$-submodules
\[0=M_0\leq \cdots\leq M_m=M\]
such that for $i=1,\ldots,m$, either $M_i/M_{i-1}\cong\nabla(\lambda_i)$ for some $\lambda_i\in X^+$ or $M_i=M_{i-1}$. If such a filtration exists, we may assume that $\lambda_i<\lambda_j$ implies $i<j$; see  Remark 4 in Section II.4.16 in \cite{Jantzen}. Moreover, the tensor product of two modules with good filtrations has a good filtration by results of S. Donkin and O. Mathieu; see \cite{DonkinGoodFiltration} and \cite{MathieuGoodFiltration}. The module $M$ is said to be \emph{contravariantly self-dual} if $M\cong M^\tau$, where $M^\tau$ denotes the contravariant dual of $M$ as in Section II.2.12 of \cite{Jantzen}. The simple $G$-modules $L(\lambda)$ are contravariantly self-dual for all $\lambda\in X^+$ and the tensor product of two contravariantly self-dual $G$-modules is contravariantly self-dual. The following proposition and corollary are also due to J. Brundan and A. Kleshchev; see Proposition 4.7 and the discussion before Lemma 4.3 in \cite{BrundanKleshchevCR}.

\begin{Proposition}[Brundan-Kleshchev] \label{prop:BrundanKleshchev}
	Suppose that $W$ is a $G$-module with a good filtration and $N$ is a contravariantly self-dual submodule of $W$. Then $N$ is completely reducible if and only if
	\[ \Hom_G(L(\delta),N) \cong \Hom_G(\Delta(\delta),N) \]
	for all $\delta\in X^+$. 
\end{Proposition}

\begin{Corollary}[Brundan-Kleshchev] \label{cor:BrundanKleshchev}
	Let $\lambda,\mu\in X^+$. Then $L(\lambda)\otimes L(\mu)$ is completely reducible if and only if
	\[ \Hom_G(L(\delta),L(\lambda)\otimes L(\mu)) \cong \Hom_G(\Delta(\delta),L(\lambda)\otimes L(\mu)) \]
	for all $\delta\in X^+$.
\end{Corollary}

The $G_1$-socles of the induced $G$-modules have been described by H.H. Andersen; see equation (4.2) in the proof of Proposition 4.1 in \cite{Andersen}:

\begin{Lemma}[Andersen] \label{lem:socleG1}
	Let $\lambda\in X^+$ and write $\lambda=\lambda_0+p\lambda_1$ with $\lambda_0\in X_1$ and $\lambda_1\in X^+$. Then
	\[ \soc_{G_1}\nabla(\lambda)\cong L(\lambda_0)\otimes \nabla(\lambda_1)^{[1]} \]
	as $G$-modules.
\end{Lemma}

As $\Delta(\lambda)\cong \nabla(\lambda)^\tau$ for all $\lambda\in X^+$ (see Section II.2.13 in \cite{Jantzen}), we have an analogous description of the $G_1$-heads of the Weyl modules:

\begin{Corollary} \label{cor:headG1}
	Let $\lambda\in X^+$ and write $\lambda=\lambda_0+p\lambda_1$ with $\lambda_0\in X_1$ and $\lambda_1\in X^+$. Then
	\[ \head_{G_1}\Delta(\lambda)\cong L(\lambda_0)\otimes \Delta(\lambda_1)^{[1]} \]
	as $G$-modules.
\end{Corollary}

We deduce the following lemma:

\begin{Lemma} \label{lem:weylgeneratedG1}
	Let $\lambda\in X_1$. Then $\soc_{G_1} \nabla(\lambda)\cong L(\lambda)\cong \head_{G_1}\Delta(\lambda)$. In addition, $\Delta(\lambda)$ is generated as a $G_1$-module by any maximal vector of weight $\lambda$. 
\end{Lemma}
\begin{proof}
	We have $\soc_{G_1} \nabla(\lambda)\cong L(\lambda)$ by Lemma \ref{lem:socleG1} and $\head_{G_1}\Delta(\lambda)\cong L(\lambda)$ by Corollary \ref{cor:headG1}. It follows that $\rad_G \Delta(\lambda)$ is the unique maximal $G_1$-submodule of $\Delta(\lambda)$. If $v\in\Delta(\lambda)$ is a maximal vector of weight $\lambda$ then $v$ is not contained in $\rad_G\Delta(\lambda)$, so $\Delta(\lambda)$ is generated as a $G_1$-module by $v$.
\end{proof}

\begin{Corollary} \label{cor:maxvecnonsimpleG1}
	Let $V$ be a rational $G$-module and let $v\in V$ be a maximal vector of $p$-restricted weight $\delta\in X_1$. If $v$ generates a simple $G_1$-submodule of $V$ then $v$ generates a simple $G$-submodule of $V$.
\end{Corollary}
\begin{proof}
	As $v$ is a maximal vector, the submodule $U_k(\mathfrak{g})\cdot v$ generated by $v$ is a homomorphic image of the Weyl module $\Delta(\delta)$. Now by Lemma \ref{lem:weylgeneratedG1}, $\Delta(\delta)$ is generated over $G_1$ by any maximal vector of weight $\delta$ and it follows that $U_k(\mathfrak{g})\cdot v = u_k(\mathfrak{g})\cdot v$. Hence, if $u_k(\mathfrak{g})\cdot v$ is a simple $G_1$-module then $U_k(\mathfrak{g})\cdot v$ is a simple $G$-module.
\end{proof}

\begin{Proposition} \label{prop:GandG1allMaxVecPRes}
	Let $\lambda,\mu\in X_1$ and suppose that all maximal vectors in $L(\lambda)\otimes L(\mu)$ have $p$-restricted weight. Then $L(\lambda)\otimes L(\mu)$ is completely reducible as a $G$-module if and only if $L(\lambda)\otimes L(\mu)$ is completely reducible as a $G_1$-module.
\end{Proposition}
\begin{proof}
	If $L(\lambda)\otimes L(\mu)$ is completely reducible as a $G$-module then $L(\lambda)\otimes L(\mu)$ is completely reducible as a $G_1$-module as the restriction of a simple $G$-module to $G_1$ is always completely reducible.
	
	Suppose that $L(\lambda)\otimes L(\mu)$ is not completely reducible as a $G$-module. By Corollary \ref{cor:BrundanKleshchev}, there exists $\delta\in X^+$ such that the natural embedding of
	$\Hom_G(L(\delta),L(\lambda)\otimes L(\mu))$ into $\Hom_G(\Delta(\delta),L(\lambda)\otimes L(\mu))$
	is not surjective and it follows that there is a maximal vector $v\in L(\lambda)\otimes L(\mu)$ of weight $\delta$ that generates a non-simple $G$-submodule of $L(\lambda)\otimes L(\mu)$. Then $\delta\in X_1$ by assumption and $v$ generates a non-simple $G_1$-submodule of $L(\lambda)\otimes L(\mu)$ by Corollary \ref{cor:maxvecnonsimpleG1}. Hence $L(\lambda)\otimes L(\mu)$ is not completely reducible as a $G_1$-module.
\end{proof}

The proof of the following lemma was suggested to the author by Stephen Donkin.

\begin{Lemma} \label{lem:DonkinNonCR}
	Let $\lambda,\mu\in X_1$ and suppose that there exists $0\neq\delta\in X^+$ such that $L(p\delta)$ is a composition factor of $L(\lambda)\otimes L(\mu)$. Then $L(\lambda)\otimes L(\mu)$ is not completely reducible as a $G_1$-module.
\end{Lemma}
\begin{proof}
	Suppose for a contradiction that $L(\lambda)\otimes L(\mu)$ is completely reducible as a $G_1$-module. Then $L(p\delta)\cong L(\delta)^{[1]}$ is a trivial $G_1$-submodule of $L(\lambda)\otimes L(\mu)$ and $\dim L(\delta)>1$ as $\delta\neq 0$. Hence
	\[ 1 < \dim \Hom_{G_1}(k,L(\lambda)\otimes L(\mu)) = \dim \Hom_{G_1}(L(\lambda)^*,L(\mu)) , \]
	contradicting Schur's Lemma.
\end{proof}

We will also need the following result due to J.-P. Serre, see Proposition 2.3 in \cite{Serre}.

\begin{Proposition}[Serre] \label{prop:Serre}
	Let $H$ be a group and let $V$ be a finite-dimensional $kH$-module such that the canonical homomorphism $k\hookrightarrow V\otimes V^*$ splits.
	If $W$ is a $kH$-module such that $V\otimes W$ is completely reducible then $W$ is completely reducible.
\end{Proposition}

\begin{Remark} \label{rem:Serre}
	\begin{enumerate}
	\item In the proofs of Propositions 2.1 and 2.3 in \cite{Serre}, we can replace the group algebra $kH$ by the finite-dimensional Hopf algebra $u_k(\mathfrak{g})$ (or any Hopf algebra, in fact), hence the above result is also valid for modules over the Frobenius kernel $G_1$.
	\item As pointed out in Remark 2.2 in \cite{Serre}, a sufficient condition for the splitting of the homomorphism $k\hookrightarrow V\otimes V^*$ is that $p$ does not divide $\dim V$. If $V$ is a simple module then
	\[ 1 = \dim \End_H(V) = \dim \Hom_H(k,V\otimes V^*) = \dim \Hom_H(V\otimes V^*,k) \]
	by Schur's lemma, so $\Hom_H(k,V\otimes V^*)$ is spanned by $x\mapsto x\cdot\id_V$ while $\Hom_H(V\otimes V^*,k)$ is spanned by the trace map, where we identify $V\otimes V^*$ with $\End_k(V)$. In that case, it follows that the embedding $k\hookrightarrow V\otimes V^*$ splits if and only if $p$ does not divide $\dim V$.
	\end{enumerate}
\end{Remark}

A finite-dimensional rational $G$-module is called \emph{multiplicity free} if all composition factors appear with multiplicity at most $1$. We now show that that multiplicity freeness of tensor products of simple modules implies complete reducibility:

\begin{Lemma} \label{lem:MultiplicityFree}
	Let $\lambda,\mu\in X^+$. If $L(\lambda)\otimes L(\mu)$ is multiplicity free then $L(\lambda)\otimes L(\mu)$ is completely reducible.
\end{Lemma}
\begin{proof}
	Suppose that $L(\lambda)\otimes L(\mu)$ is not completely reducible so that $\rad_G\big(L(\lambda)\otimes L(\mu)\big)\neq 0$. Let~$L$ be a simple submodule of $\soc_G\big(\rad_G\big(L(\lambda)\otimes L(\mu)\big)\big)\subseteq \soc_G\big( L(\lambda)\otimes L(\mu) \big)$. As $L(\lambda)\otimes L(\mu)$ is contravariantly self-dual, we have
	\[ \soc_G\big( L(\lambda)\otimes L(\mu) \big) \cong \head_G\big( L(\lambda)\otimes L(\mu) \big) = \big( L(\lambda)\otimes L(\mu) \big) / \rad_G\big( L(\lambda)\otimes L(\mu) \big) . \]
	Thus $L$ is a composition factor of $\rad_G\big( L(\lambda)\otimes L(\mu) \big)$ and of $\head_G\big( L(\lambda)\otimes L(\mu) \big)$, so $L(\lambda)\otimes L(\mu)$ is not multiplicity free.
\end{proof}

We conclude the section with a remark about truncation to Levi subgroups.

\begin{Remark} \label{rem:levi}
	Let $I\subseteq \{1,\ldots, n\}$ and consider the derived subgroup $L_I$ of the Levi subgroup of $G$ corresponding to $\Delta_I=\{\alpha_i\mid i\in I\}$. For $\lambda=a_1\omega_1+\cdots+a_n\omega_n\in X^+$, write $\lambda_I=\sum_{i\in I}a_i\omega_i$. For a rational $G$-module $M$ and $\mu\in X$, we define the \emph{truncation} of $M$ to $L_I$ at $\mu$ by
	\[ \mathrm{Tr}_I^\mu M = \bigoplus_{\delta\in \Z \Delta_I} M_{\mu-\delta} . \]
	 Then for $\lambda\in X^+$, $\mathrm{Tr}_I^\lambda L(\lambda)$ is the simple $L_I$-module of highest weight $\lambda_I$, see Sections  II.2.10 and~II.2.11 in \cite{Jantzen}. Analogously, $\mathrm{Tr}_I^\lambda \nabla(\lambda)$ and $\mathrm{Tr}_I^\lambda \Delta(\lambda)$ afford the induced module and the Weyl module of highest weight $\lambda_I$ for $L_I$. If $M$ is any finite-dimensional rational $G$-module of highest weight $\lambda\in X^+$, then it is straightforward to check that for $\mu\in (\lambda-\Z\Delta_I)\cap X^+$, the multiplicity of $L(\mu)$ as a composition factor of $M$ coincides with the multiplicity of the simple $L_I$-module of highest weight $\mu_I$ in $\mathrm{Tr}_I^\lambda M$. Furthermore, for $\lambda,\mu\in X^+$,
	\[ \mathrm{Tr}_I^{\lambda+\mu}\big( L(\lambda)\otimes L(\mu) \big) = \mathrm{Tr}_I^\lambda L(\lambda)\otimes \mathrm{Tr}_I^\mu L(\mu) \]
	is the tensor product of the simple $L_I$-modules of highest weights $\lambda_I$ and $\mu_I$. In particular, the latter tensor product is completely reducible whenever $L(\lambda)\otimes L(\mu)$ is completely reducible. This observation will be crucial in Sections \ref{sec:Anp2} and  \ref{sec:smallprimes}.
\end{Remark}

\section{Complete reducibility and \texorpdfstring{$p$}{p}-restrictedness} \label{sec:CRprestricted}

We are now ready to prove Theorems \ref{thm:introA} and \ref{thm:introB} for $G$ of type different from $\mathrm{G}_2$ and under the assumption that $p>2$ when $G$ is of type $\mathrm{B}_n$, $\mathrm{C}_n$ or $\mathrm{F}_4$.

\begin{Theorem} \label{thm:completelyreducibleprestricted}
	Assume that $\Phi$ is of type different from $\mathrm{G}_2$ and that $p>2$ if $\Phi$ is not simply laced. Let $\lambda,\mu\in X_1$. If $L(\lambda)\otimes L(\mu)$ is completely reducible then all composition factors of $L(\lambda)\otimes L(\mu)$ are $p$-restricted.
\end{Theorem}
\begin{proof}
	Suppose that $L(\lambda)\otimes L(\mu)$ is completely reducible, so completely reducible as a $G_1$-module, and let $\delta\in X^+$ such that $L(\delta)$ is a composition factor of $L(\lambda)\otimes L(\mu)$. Then $L(\delta)$ is generated by a maximal vector $v\in L(\lambda)\otimes L(\mu)$ of weight $\delta$ and $\delta\in X_1^\prime$ by Corollary \ref{cor:nonpresnonsimpleG1}, hence $\delta\in X_1^\prime \cap X^+=X_1$.
\end{proof}

\begin{Theorem} \label{thm:completelyreducibleGandG1}
	Assume that $\Phi$ is of type different from $\mathrm{G}_2$ and that $p>2$ if $\Phi$ is not simply laced. Let $\lambda,\mu\in X_1$. Then $L(\lambda)\otimes L(\mu)$ is completely reducible as a $G$-module if and only if $L(\lambda)\otimes L(\mu)$ is completely reducible as a $G_1$-module.
\end{Theorem}
\begin{proof}
	If $L(\lambda)\otimes L(\mu)$ is completely reducible as a $G$-module then $L(\lambda)\otimes L(\mu)$ is completely reducible as a $G_1$-module as the restriction of a simple $G$-module to $G_1$ is always completely reducible.
	
	Now suppose that $L(\lambda)\otimes L(\mu)$ is completely reducible as a $G_1$-module. Then by Corollary \ref{cor:nonpresnonsimpleG1}, all maximal vectors in $L(\lambda)\otimes L(\mu)$ have $p$-restricted weight and $L(\lambda)\otimes L(\mu)$ is completely reducible as a $G$-module by Proposition \ref{prop:GandG1allMaxVecPRes}.
\end{proof}

\begin{Remark} \label{rem:afterABforCnp2}
	Note that Proposition \ref{prop:GandG1allMaxVecPRes} is valid in arbitrary characteristic and for all types of root systems. In order to prove Theorems \ref{thm:introA} and \ref{thm:introB} for $\Phi$ and $p$ that are not included in the above statements, it would be sufficient to obtain an analogue of Corollary \ref{cor:nonpresnonsimpleG1} for the corresponding group~$G$. We will do this for $G$ of type $\mathrm{C}_n$ and $p=2$ in Section \ref{sec:smallprimes}; see Proposition \ref{prop:Cnp2nonCRG1} and Remark \ref{rem:Cnp2thmAB}.
\end{Remark}

\section{Results for \texorpdfstring{$G$}{G} of type \texorpdfstring{$\mathrm{G}_2$}{G2} and \texorpdfstring{$p\neq 3$}{p<>3}}  \label{sec:G2}

In this section, we prove Theorems \ref{thm:introA} and \ref{thm:introB} for $G$ of type $\mathrm{G}_2$ when $p\neq 3$.

Let $G$ be of type $\mathrm{G}_2$ and let the simple roots $\Delta=\{\alpha_1,\alpha_2\}$ be ordered such that $\alpha_1$ is short, that is $\langle \alpha_1,\alpha_2^\vee \rangle=-1$ and $\langle \alpha_2,\alpha_1^\vee \rangle=-3$. The highest root in $\Phi$ is given by $\alpha_0=3\alpha_1+2\alpha_2$ and we have $\langle\alpha_0 ,\alpha_0\rangle=\langle\alpha_2,\alpha_2\rangle=3\cdot\langle\alpha_1,\alpha_1\rangle$. For $\lambda\in X$, it follows that
	\[ \langle \lambda ,\alpha_0^\vee \rangle = \frac{2\langle\lambda,\alpha_0\rangle}{\langle\alpha_0,\alpha_0\rangle} = \langle \lambda,\alpha_1^\vee \rangle + 2\cdot  \langle \lambda,\alpha_2^\vee \rangle . \]

\begin{Lemma} \label{lem:G2inequality}
	Suppose that $p>3$, let $\lambda,\mu\in X_1$ and write $\lambda=a\omega_1+b\omega_2$ and $\mu=c\omega_1+d\omega_2$. If $L(\lambda)\otimes L(\mu)$ is completely reducible as a $G_1$-module then $a+c+3\cdot\min\{b,d\}<p$ and $b+d+\min\{a,c\}<p$. In particular, we have either $\langle\lambda,\alpha_0^\vee\rangle<p$ or $\langle\mu,\alpha_0^\vee\rangle<p$.
\end{Lemma}
\begin{proof}
	The $\mathrm{SL}_2(k)$-module $L(b)\otimes L(d)$ has composition factors of highest weights $b+d-2r$ for all $0\leq r\leq \min\{b,d\}$ and by Remark \ref{rem:levi}, $L(\lambda)\otimes L(\mu)$ has composition factors of highest weights $\lambda+\mu-r\alpha_2$ for $0\leq r\leq \min\{b,d\}$. Hence by Proposition \ref{prop:nonpreslevel0},
	\[ \langle \lambda+\mu-r\alpha_2 , \alpha_1^\vee \rangle = a+c - r\cdot \langle \alpha_2 , \alpha_1^\vee \rangle = a+c+3r \]
	must not take values strictly between $p-1$ and $2p-1$ for $0\leq r\leq \min\{b,d\}$. As $p>3$, it follows that
	\[ a + c + 3\cdot\min\{b,d\}<p . \]
	Analogously, we obtain that $b+d+\min\{a,c\}<p$ as $\langle\alpha_1,\alpha_2^\vee\rangle=-1$. The first inequality yields that either $ \langle \lambda , \alpha_0^\vee \rangle = a+2b<p$ or $\langle \mu , \alpha_0^\vee \rangle = c+2d<p$ and hence the final claim.
\end{proof}

For $p\in\{2,5,7\}$, the proofs of the following theorems rely on computations which were carried out in GAP \cite{GAP4}. For a fixed prime $p$ and root system $\Phi$ (of small rank), S. Doty's WeylModules-package, available on his website \cite{Doty}, enables us to compute the characters of simple modules, the composition factors of Weyl modules and the composition factors of tensor products of simple modules. Let $G_\C$ be the simply connected simple algebraic group with root system $\Phi$ over $\C$. Then the multiplicity of an induced module $\nabla(\delta)$ in a good filtration of $\nabla(\lambda)\otimes\nabla(\mu)$ coincides with the multiplicity of the simple $G_\C$-module $\nabla_\C(\delta)$ in the tensor product $\nabla_\C(\lambda)\otimes\nabla_\C(\mu)$, and the latter can be computed in GAP. Note that the multiplicity of $\nabla(\delta)$ in a good filtration of $\nabla(\lambda)\otimes\nabla(\mu)$ is also given by $\dim \Hom_G\big(\Delta(\delta),\nabla(\lambda)\otimes\nabla(\mu)\big)$; see Proposition II.4.16 in \cite{Jantzen}.

\begin{Theorem} \label{thm:completelyreducibleprestrictedG2}
	Assume that $\Phi$ is of type $\mathrm{G}_2$ and $p\neq 3$. Let $\lambda,\mu\in X_1$. If $L(\lambda)\otimes L(\mu)$ is completely reducible then all composition factors of $L(\lambda)\otimes L(\mu)$ are $p$-restricted.
\end{Theorem}
\begin{proof}
	Suppose that $L(\lambda)\otimes L(\mu)$ is completely reducible. We may assume that $\lambda\neq 0$ and $\mu\neq 0$ and that $\lambda+\mu\in X_1$ by Corollary \ref{cor:highestweightvectorprestricted}. If $p=2$, it follows that $\{\lambda,\mu\}=\{\omega_1,\omega_2\}$ and we can compute that $L(2\omega_2)$ is a composition factor of $L(\omega_1)\otimes L(\omega_2)$, hence $L(\omega_1)\otimes L(\omega_2)$ is not completely reducible as a $G_1$-module by Lemma \ref{lem:DonkinNonCR}.
	
	If $p>3$ then we have either $ \langle \lambda , \alpha_0^\vee \rangle < p $ or $ \langle \mu , \alpha_0^\vee \rangle < p $ by Lemma \ref{lem:G2inequality} and the claim is immediate from Theorem \ref{thm:BrundanKleshchevsocle}.
\end{proof}

\begin{Theorem} \label{thm:completelyreducibleGandG1forG2}
	Assume that $\Phi$ is of type $\mathrm{G}_2$ and $p\neq 3$. Let $\lambda,\mu\in X_1$. Then $L(\lambda)\otimes L(\mu)$ is completely reducible as a $G$-module if and only if $L(\lambda)\otimes L(\mu)$ is completely reducible as a $G_1$-module.
\end{Theorem}
\begin{proof}
	As before, if $L(\lambda)\otimes L(\mu)$ is completely reducible as a $G$-module then $L(\lambda)\otimes L(\mu)$ is completely reducible as a $G_1$-module.
	Suppose for a contradiction that $L(\lambda)\otimes L(\mu)$ is completely reducible as a $G_1$-module, but not as a $G$-module. For $p=2$, the statement follows as in the proof of Theorem \ref{thm:completelyreducibleprestrictedG2}, so now assume that~$p>3$. By Lemma \ref{lem:G2inequality}, we may assume without loss of generality that $\langle\lambda,\alpha_0^\vee\rangle<p$ so that the socle of $L(\lambda)\otimes L(\mu)$ is $p$-restricted by Theorem \ref{thm:BrundanKleshchevsocle}. By Corollary \ref{cor:BrundanKleshchev}, there exists $\delta\in X^+$ such that the natural embedding of
	$\Hom_G(L(\delta),L(\lambda)\otimes L(\mu))$ into $\Hom_G(\Delta(\delta),L(\lambda)\otimes L(\mu))$
	is not surjective.
	Hence there exists a maximal vector $v\in L(\lambda)\otimes L(\mu)$ of weight $\delta$ that generates a non-simple $G$-submodule $U=U_k(\mathfrak{g})\cdot v$. As $L(\lambda)\otimes L(\mu)$ is completely reducible as a $G_1$-module, $v$ generates a simple $G_1$-submodule and it follows that $\delta$ is non-$p$-restricted by Corollary \ref{cor:maxvecnonsimpleG1}. Now $U$ is a quotient of $\Delta(\delta)$ and $U\vert_{G_1}$ is completely reducible, hence $U$ is a quotient of $\head_{G_1}\Delta(\delta)\cong L(\delta_0)\otimes \Delta(\delta_1)^{[1]}$, where $\delta=\delta_0+p\delta_1$ with $\delta_0\in X_1$ and $\delta_1\in X^+$ (see Corollary \ref{cor:headG1}). Furthermore, $\soc_G U$ is $p$-restricted as $\soc_G(L(\lambda)\otimes L(\mu))$ is $p$-restricted and it follows that $L(\delta_0)\otimes \Delta(\delta_1)^{[1]}$ has a $p$-restricted composition factor. Hence the trivial module $L(0)$ is a composition factor of $\Delta(\delta_1)$.
	
	If $\delta=\lambda+\mu-x\alpha_1 -y\alpha_2$ then
	\[ \langle \delta , \alpha_0^\vee \rangle = \langle \lambda , \alpha_0^\vee \rangle + \langle \mu , \alpha_0^\vee \rangle - x \langle\alpha_1,\alpha_0^\vee\rangle - y \langle \alpha_2 , \alpha_0^\vee \rangle = \langle \lambda , \alpha_0^\vee \rangle + \langle \mu , \alpha_0^\vee \rangle - y . \]
	Now $\langle \lambda , \alpha_0^\vee \rangle<p$ by assumption and $\langle \mu , \alpha_0^\vee \rangle = \langle \mu , \alpha_1^\vee \rangle + 2 \langle \mu , \alpha_2^\vee \rangle<3p$ as $\mu$ is $p$-restricted. It follows that $\langle \delta , \alpha_0^\vee \rangle<4p$ and $\langle \delta_1 , \alpha_0^\vee \rangle<4$. Thus $\delta_1=a\omega_1+b\omega_2$ with $a+2b<4$, that is
	\[ (a,b)\in \{(1,0),(2,0),(3,0),(0,1),(1,1)\} . \]
	Let $\tilde{\alpha}_0\in\Phi$ be the highest short root. If $p>7$ then $\langle\delta_1+\rho,\tilde{\alpha}_0^\vee\rangle=2a+3b+5\leq p$, so $\delta_1$ lies in the closure of the bottom alcove and $\Delta(\delta_1)$ is simple by the linkage principle, a contradiction. If~$p=5$ then again $\Delta(\delta_1)$ is simple, as we can compute using GAP \cite{GAP4}.
	Finally, if $p=7$ then the unique dominant weight $\delta_1$ with $\langle \delta_1 , \alpha_0^\vee \rangle<4$ such that $L(0)$ is a composition factor of $\Delta(\delta_1)$ is $\delta_1=2\omega_1$.
	Thus $\delta=\delta_0+14\omega_1$ and so~$\langle\delta,\alpha_1^\vee\rangle\geq 14$. As $\delta\leq\lambda+\mu$ and $\lambda+\mu$ is $7$-restricted by Corollary \ref{cor:highestweightvectorprestricted}, it follows that $\lambda+\mu$ is one of the weights $5\omega_1+6\omega_2$ and $6\omega_1+6\omega_2$. Using the inequalities from Lemma~\ref{lem:G2inequality}, we conclude that either~$\{\lambda,\mu\}=\{5\omega_1,6\omega_2\}$ or $\{\lambda,\mu\}=\{6\omega_1,6\omega_2\}$. Now by decomposing the character of $\nabla(\lambda)\otimes\nabla(\mu)$ using GAP in both cases, we find that
	\[ 0 = \Hom_G(\Delta(\delta^\prime),\nabla(\lambda)\otimes \nabla(\mu)) \supseteq \Hom_G(\Delta(\delta^\prime),L(\lambda)\otimes L(\mu)) \]
	for all $\delta^\prime\in X^+$ with $\langle\delta^\prime,\alpha_1^\vee\rangle\geq 14$, a contradiction.
\end{proof}

\section{Type \texorpdfstring{$\mathrm{A}_n$}{An} in characteristic \texorpdfstring{$2$}{2}} \label{sec:Anp2}

In this section, we classify all pairs of $p$-restricted weights $\lambda$ and $\mu$ such that $L(\lambda)\otimes L(\mu)$ is completely reducible for $G$ of type $\mathrm{A}_n$ when $p=2$.

Let $G=\SL_{n+1}(k)$ and denote by $V=k^{n+1}$ the natural $n+1$-dimensional $\GL_{n+1}(k)$-module with dual module $V^*$. We have $V=L(\omega_1)$ and $V^*=L(\omega_n)$ as $G$-modules. In  \cite{BrundanKleshchevCR}, J. Brundan and A. Kleshchev classify the simple $\GL_{n+1}(k)$-modules $L$ such that $V\otimes L$ is completely reducible, in arbitrary characteristic $p$. We recall their result:

The dominant weights for $\GL_{n+1}(k)$ can be identified with the set of $n+1$-tuples $\lambda=(\lambda_1,\ldots,\lambda_{n+1})$ of integers with $\lambda_1\geq\cdots\geq\lambda_{n+1}$ and we write $L^\prime(\lambda)$ and $\Delta^\prime(\lambda)$ for the corresponding simple module and the corresponding Weyl module, respectively (to distinguish from our notation for $G=\SL_{n+1}(k)$). After tensoring with a power of the one-dimensional determinant representation of $\GL_{n+1}(k)$, we may assume that $\lambda_n\geq 0$, so it is sufficient to consider the simple modules $L^\prime(\lambda)$ for $\lambda$ a partition with $n+1$ parts.
Now fix a partition $\lambda=(\lambda_1,\ldots,\lambda_{n+1})$. For $i\in\{1,\ldots,n+1\}$, we say that $i$ is $\lambda$-addable if $i=1$ or $\lambda_i<\lambda_{i-1}$, and $\lambda$-removable if $i=n+1$ or $\lambda_{i+1}<\lambda_i$. For $a,b\in\Z$, denote by $\res(a,b)$ the residue of $b-a$ in $\Z/p\Z$. For $i\in\{1,\ldots,n+1\}$, we define the following sets:
\begin{align*}
	M_\lambda^-(i) & =\{  j<i \mid j\text{ is }\lambda \text{-removable and } \res(j,\lambda_j)=\res(i,\lambda_i+1)  \} \\
	M_\lambda^+(i) & =\{  j<i \mid j\text{ is }\lambda \text{-addable and } \res(j,\lambda_j+1)=\res(i,\lambda_i+1)  \}
\end{align*}
Then $i\in\{1,\ldots,n+1\}$ is called $\lambda$-conormal if $i$ is $\lambda$-addable and there is an increasing injection
\[\varphi\colon M_\lambda^-(i)\hookrightarrow M_\lambda^+(i) , \]
that is, an injection with $\varphi(j)>j$ for all $j\in M_\lambda^-(i)$. We say that $i$ is $\lambda$-cogood if $i$ is $\lambda$-conormal and maximal among the $\lambda$-conormal $j$ with $\res(j,\lambda_j+1)=\res(i,\lambda_i+1)$. Denote by $\varepsilon_i$ the $n+1$-tuple $(0,\ldots,0,1,0,\ldots,0)$ with entry $1$ in the $i$-th position.

The following result combines Theorem 5.11 and Corollary 5.12 in \cite{BrundanKleshchevCR}.

\begin{Theorem}[Brundan-Kleshchev] \label{thm:BrundanKleshchevTensorNatural}
	Let $\lambda$ and $\mu$ be partitions with $n+1$ parts.
	\begin{enumerate}
		\item The space $\Hom_{\GL_{n+1}(k)}(\Delta^\prime(\mu),V\otimes L^\prime(\lambda))$ is zero unless $\mu=\lambda+\varepsilon_i$ for some $\lambda$-conormal $i$, in which case it is $1$-dimensional.
		\item The space $\Hom_{\GL_{n+1}(k)}(L^\prime(\mu),V\otimes L^\prime(\lambda))$ is zero unless $\mu=\lambda+\varepsilon_i$ for some $\lambda$-cogood $i$, in which case it is $1$-dimensional.
		\item $V\otimes L^\prime(\lambda)$ is completely reducible if and only if every $\lambda$-conormal $i$ is $\lambda$-cogood.
	\end{enumerate}
\end{Theorem}

Now for a partition $\lambda=(\lambda_1,\ldots,\lambda_{n+1})$, the restriction of the simple module $L^\prime(\lambda)$ to $G=\SL_{n+1}(k)$ is simple of highest weight
\[ \lambda^\prime \coloneqq (\lambda_1-\lambda_2)\cdot \omega_1 + \cdots + (\lambda_n-\lambda_{n+1})\cdot \omega_n \in X^+ \]
and the $\GL_{n+1}(k)$-module $V\otimes L^\prime(\lambda)$ is completely reducible if and only if the $G$-module $V\otimes L(\lambda^\prime)$ is completely reducible. For a dominant weight $\mu=a_1\omega_1+\cdots+a_n\omega_n\in X^+$, we define a partition $\pi(\mu)=(\lambda_1,\ldots,\lambda_{n+1})$ by $\lambda_i=a_1+\cdots + a_{n+1-i}$ for $i<{n+1}$ and $\lambda_{n+1}=0$. Then clearly $\pi(\mu)^\prime=\mu$, so the $G$-module $V\otimes L(\mu)$ is completely reducible if and only if every $\pi(\mu)$-conormal $i$ is $\pi(\mu)$-cogood. Also note that $(\pi(\mu)+\varepsilon_i)^\prime=\mu-\omega_{i-1}+\omega_i$ for all $\lambda$-addable $i$, where we take $\omega_0$ and $\omega_{n+1}$ to be $0$.

The Loewy length of a rational $G$-module $M$ is defined to be the minimal length of a filtration with completely reducible quotients.

\begin{Corollary} \label{cor:LoewyLengthAnp2}
	Let $\mu\in X^+$. If $V\otimes L(\mu)$ is not completely reducible then $V\otimes L(\mu)$ has Loewy length at least $3$.
\end{Corollary}
\begin{proof}
	Suppose that $V\otimes L(\mu)$ is not completely reducible, so $V\otimes L(\mu)$ has Loewy length at least $2$. By Theorem \ref{thm:BrundanKleshchevTensorNatural}, there is a $\pi(\mu)$-conormal $i\in\{1,\ldots,n\}$ that is not $\pi(\mu)$-cogood and we have
	\[ \Hom_G(\Delta(\mu-\omega_{i-1}+\omega_i),V\otimes L(\mu))\neq 0 \qquad \text{and} \qquad \Hom_G(L(\mu-\omega_{i-1}+\omega_i),V\otimes L(\mu))=0 . \]
	If $V\otimes L(\mu)$ has Loewy length $2$ then $L(\mu-\omega_{i-1}+\omega_i)$ belongs to $\head_G(V\otimes L(\mu))$, hence
	\[ 0 \neq \Hom_G(V\otimes L(\mu),L(\mu-\omega_{i-1}+\omega_i)) \cong \Hom_G(L(\mu-\omega_{i-1}+\omega_i),V\otimes L(\mu)) \]
	by contravariant duality, a contradiction. Hence $V\otimes L(\mu)$ has Loewy length at least $3$.
\end{proof}

For $p=2$, we can derive from Theorem \ref{thm:BrundanKleshchevTensorNatural} a simple characterization of the $2$-restricted weights $\mu$ such that $V\otimes L(\mu)$ is completely reducible.
\medskip

\noindent
\textbf{Hypothesis.} For the rest of the section, suppose that $p=2$.

\begin{Proposition} \label{prop:BrundanKleshchevGood}
	Let $\mu\in X_1$. Then $V\otimes L(\mu)$ is completely reducible if and only if $\mu=\omega_{i_1}+\cdots+\omega_{i_r}$ for even numbers $i_1,\ldots,i_r$ with $1<i_1<\cdots<i_r\leq n$.  
\end{Proposition}
\begin{proof}
	As $\mu$ is $2$-restricted, we have $\mu=\omega_{i_1}+\cdots+\omega_{i_r}$ for certain $1\leq i_1<\cdots<i_r\leq n$. Let us write
	\[ \lambda=(\lambda_1,\ldots,\lambda_{n+1}) \coloneqq \pi(\mu)=(r^{i_1},(r-1)^{i_2-i_1},\ldots,1^{i_r-i_{r-1}},0^{n+1-i_r}) , \]
	where the notation $a^s$ stands for an $s$-tuple $(a,\ldots,a)$.
	The $\lambda$-addable indices are $\{1,i_1+1,\ldots,i_r+1 \}$ and the $\lambda$-removable indices are $\{i_1,\ldots,i_r,n+1\}$. Note that $M_\lambda^-(1)=\varnothing$, so $1$ is $\lambda$-conormal. Furthermore, we have
	\[ \res(i_1,\lambda_{i_1})=\overline{r-i_1}\neq \overline{r-i_1-1}= \res(i_1+1,\lambda_{i_1+1}+1) , \]
	where we denote by $\overline{a}$ the image of $a\in\Z$ in $\Z/2\Z$, and it follows that $M_\lambda^-(i_1+1)=\varnothing$, so $i_1+1$ is $\lambda$-conormal.
	
	Assume first that $V\otimes L(\mu)$ is completely reducible so that every $\pi(\mu)$-conormal $i$ is $\pi(\mu)$-cogood. Suppose for a contradiction that $i_1,\ldots,i_r$ are not all even and let $s\in\{1,\ldots,r\}$ be minimal with the property that $i_s$ is odd. If $s=1$ then
	\[ \res(1,\lambda_1+1)=\overline{r}=\overline{r-i_1-1}=\res(i_1+1,\lambda_{i_1+1}+1) \]
	and $1$ is not $\lambda$-cogood, a contradiction. Hence $s>1$. We have
	\[ \res(i_s+1,\lambda_{i_s+1}+1)=\overline{(r+1-s)-(i_s+1)}=\overline{r+1-s} \]
	as $i_s$ is odd and for $j<s$,
	\[ \res(i_j,\lambda_{i_j})=\overline{(r+1-j)-i_j}=\overline{r+1-j} \]
	as $i_j$ is even, so $i_j\in M_\lambda^-(i_s+1)$ if and only if $\overline{s-j}=0$. Furthermore, we have
	\[ \res(i_s,\lambda_{i_s}) = \overline{(r+1-s)-i_s}=\overline{r-s}\neq \overline{r+1-s} , \]
	so $i_s\notin M_\lambda^-(i_s+1)$. 
	Analogously, we have
	\[ \res(i_j+1,\lambda_{i_j+1}+1)=\overline{(r+1-j)-(i_j+1)}=\overline{r-j} \]
	for $j<s$ and therefore $i_j+1\in M_\lambda^+(i_s+1)$ if and only if $\overline{s-j}\neq 0$. It follows that for all $j<s$ with $i_j\in M_\lambda^-(i_s+1)$, we have $i_{j+1}+1\in M_\lambda^+(i_s+1)$ and $i_j\mapsto i_{j+1}+1$ defines an increasing injection
	\[M_\lambda^-(i_s+1)\hookrightarrow M_\lambda^+(i_s+1)\]
	so $i_{s}+1$ is $\lambda$-conormal. Now $1$ and $i_1+1$ are $\lambda$-conormal with
	\[\res(1,\lambda_1+1)=\overline{r}\qquad \text{and} \qquad \res(i_1+1,\lambda_{i_1+1}+1)=\overline{r-i_1-1}=\overline{r-1}\]
	the two distinct elements of $\Z/2\Z$. Hence either $\res(i_s+1,\lambda_{i_s+1}+1)=\res(i_1+1,\lambda_{i_1+1}+1)$ or $\res(i_s+1,\lambda_{i_s+1}+1)=\res(1,\lambda_1+1)$ and it follows that one of $1$ and $i_1+1$ is not $\lambda$-cogood, a contradiction.
	
	Now suppose that $i_1,\ldots,i_r$ are even. As before, $1$ and $i_1+1$ are $\lambda$-conormal with
	\[ \res(1,\lambda_1+1)\neq\res(i_1+1,\lambda_{i_1+1}+1) . \]
	We show that $i_t+1$ is not $\lambda$-conormal for $t>1$. Indeed, we have
	\[ \res(i_j+1,\lambda_{i_j+1}+1)=\overline{(r+1-j)-(i_j+1)}=\overline{r-j} \]
	and
	\[ \res(i_j,\lambda_{i_j})=\overline{(r+1-j)-i_j}=\overline{r+1-j} \]
	for $1\leq j\leq r$, so
	\[ M_\lambda^-(i_t+1)=\{ i_j \mid j<t \text{ and } \overline{t-j}\neq 0 \} \quad\text{and}\quad M_\lambda^+(i_t+1)\subseteq \{ i_j+1 \mid j<t \text{ and } \overline{t-j}=0 \}\cup\{1\} . \]
	Hence the maximal element of $M_\lambda^-(i_t+1)$ is $i_{t-1}$ while the maximal element of $M_\lambda^+(i_t+1)$ is $i_{t-2}+1$ if $t>2$ and $1$ otherwise. Now as $1<i_1$ and $i_{t-1}-i_{t-2}>0$ is even if $t>2$, there does not exist an increasing injection from $M_\lambda^-(i_t+1)$ to $M_\lambda^+(i_t+1)$ and $i_t+1$ is not $\lambda$-conormal.
	
	We conclude that $1$ and $i_1+1$ are $\lambda$-cogood, so $V\otimes L(\mu)$ is completely reducible by Theorem~\ref{thm:BrundanKleshchevTensorNatural}.
\end{proof}

\begin{Remark} \label{rem:Anp2}
	We note some more consequences of the proof of Proposition \ref{prop:BrundanKleshchevGood}. Let $\mu=\omega_{i_1}+\cdots+\omega_{i_r}$ with $1<i_1<\ldots<i_r\leq n$.
	\begin{enumerate}
		\item If $i_1,\ldots,i_r$ are even then $V\otimes L(\mu)$ is completely reducible and $1$ and $i_1+1$ are the unique $\pi(\mu)$-cogood indices. It follows from Theorem \ref{thm:BrundanKleshchevTensorNatural} that
		\[ V\otimes L(\mu) \cong L(\omega_1+\mu)\oplus L(\mu-\omega_{i_1}+\omega_{i_1+1}) . \]
		\item If $i_1,\ldots,i_r$ are not all even and $s\in\{1,\ldots, r\}$ is minimal with the property that $i_s$ is odd then $i_s+1$ is $\pi(\mu)$-conormal. If $s=r$, it follows that $i_r+1$ is $\pi(\mu)$-cogood and
		\[ \Hom_G(L(\mu-\omega_{i_r}+\omega_{i_r+1}),V\otimes L(\mu)) \neq 0 , \]
		where we take $\omega_{n+1}$ to be $0$.
	\end{enumerate}
\end{Remark}

\begin{Lemma} \label{lem:Anp2highestroot}
	Let $\mu\in X_1\setminus \{0\}$. Then $L(\omega_1+\omega_n)\otimes L(\mu)$ is not completely reducible.
\end{Lemma}
\begin{proof}
	Suppose for a contradiction that $L(\omega_1+\omega_n)\otimes L(\mu)$ is completely reducible. As $\mu\in X_1$, we have $\mu=\omega_{i_1}+\cdots+\omega_{i_r}$ for certain $i_1<\ldots<i_r$, where $1<i_1$ and~$i_r<n$ by Corollary \ref{cor:highestweightvectorprestricted}. Moreover, the truncation of $L(\omega_1+\omega_n)\otimes L(\mu)$ to the two Levi subgroups of type~$\mathrm{A}_{n-1}$ is completely reducible by Remark \ref{rem:levi} and it follows from Proposition \ref{prop:BrundanKleshchevGood} that $i_1,\ldots,i_r$ and $n+1-i_1,\ldots,n+1-i_r$ are even. Hence $n+1$ is even, so $V\otimes V^*$ is indecomposable of composition length $3$, with a unique composition series
	\[ 0\leq V_1\leq V_2\leq V_3=V\otimes V^* , \]
	where $V_1=\soc_G(V\otimes V^*)$ and $V_2=\rad_G(V\otimes V^*)$ and we have
	$V_1\cong V_3/V_2\cong k$ and $V_2/V_1 \cong L(\omega_1+\omega_n)$. 
	Then
	\[ 0\leq V_1\otimes L(\mu) \leq V_2\otimes L(\mu) \leq V_3 \otimes L(\mu) = V\otimes V^*\otimes L(\mu) \]
	is a filtration with completely reducible quotients and it follows that $V\otimes V^*\otimes L(\mu)$ has Loewy length at most $3$, and $V\otimes V^*\otimes L(\mu)$ has at most $1$ indecomposable direct summand of Loewy length $3$. Indeed, if
	\[ V\otimes V^*\otimes L(\mu) = M_1 \oplus \cdots \oplus M_t \]
	for indecomposable $G$-modules $M_1,\ldots,M_t$ then the intersections $M_{ij}\coloneqq M_i\cap (V_j\otimes L(\mu))$ for $j=1,2,3$ afford a filtration of $M_i$ with completely reducible quotients. However, as $V_1\otimes L(\mu)\cong L(\mu)$ is simple, there is at most one $i\in\{1,\ldots,t\}$ such that $M_{i1}=M_i\cap(V_1\otimes L(\mu))$ is non-trivial, hence $M_j$ has Loewy length at most $2$ for $j\neq i$.
	
	Now $V\otimes L(\mu)\cong L(\mu+\omega_1)\oplus L(\mu-\omega_{i_1}+\omega_{i_1+1})$ by Remark \ref{rem:Anp2} and neither of
	\[ V^*\otimes L(\mu+\omega_1) \qquad\text{and}\qquad V^*\otimes L(\mu-\omega_{i_1}+\omega_{i_1+1}) \]
	is completely reducible by Proposition \ref{prop:BrundanKleshchevGood} and duality, as $n$ and $n+1-(i_1+1)$ are odd. Thus
	\[V\otimes V^*\otimes L(\mu) \cong \big( V^*\otimes L(\mu+\omega_1) \big) \oplus \big( V^*\otimes L(\mu-\omega_{i_1}+\omega_{i_1+1}) \big) \]
	has at least two indecomposable direct summands of Loewy length at least $3$ by Corollary~\ref{cor:LoewyLengthAnp2}, a contradiction. 
\end{proof}

\begin{Lemma} \label{lem:Anp2omega2}
	Let $n\geq 4$ and $2<i< n$. Then $L(\omega_2)\otimes L(\omega_i+\omega_n)$ is not completely reducible.
\end{Lemma}
\begin{proof}
	Suppose for a contradiction that $L(\omega_2)\otimes L(\omega_i+\omega_n)$ is completely reducible. By truncating to~$L_{[2,n]}$ and applying Proposition \ref{prop:BrundanKleshchevGood}, we see that $i$ and $n$ are odd.
	The module $V\otimes V$ has a unique composition series $0=V_0 \leq V_1\leq V_2\leq V_3=V\otimes V$, where
	\[ V_1=\soc_G(V\otimes V)\qquad \text{and}\qquad  V_2=\rad_G(V\otimes V) , \]
	and we have $V_1\cong V_3/V_2\cong L(\omega_2)$ and $V_2/V_1\cong L(2\omega_1)\cong L(\omega_1)^{[1]}$. Then
	\[ 0 \leq V_1\otimes L(\omega_i+\omega_n) \leq V_2\otimes L(\omega_i+\omega_n) \leq V_3\otimes L(\omega_i+\omega_n) = V\otimes V \otimes L(\omega_i+\omega_n) \eqqcolon M \]
	is a filtration of $M$ with completely reducible quotients, so the Loewy length of $M$ is at most $3$. Furthermore, the middle layer of this filtration
	\[ V_2\otimes L(\omega_i+\omega_n) / V_1\otimes L(\omega_i+\omega_n) \cong L(\omega_1)^{[1]}\otimes L(\omega_i+\omega_n)\cong L(2\omega_1+\omega_i+\omega_n) \]
	is simple and $(\rad_G M+\soc_G M)/\soc_G M$ embeds into $L(2\omega_1+\omega_i+\omega_n)$. We show that $i+1$ is $\pi(\omega_i+\omega_n)$-cogood. Indeed,
	\[ \pi(\omega_i+\omega_n)=(2^i,1^{n-i},0)\eqqcolon \lambda \]
	has addable indices $\{1,i+1,n+1\}$ and removable indices $\{i,n,n+1\}$ and $i+1$ is $\lambda$-conormal by Remark~\ref{rem:Anp2}. Furthermore, $\res(n+1,1)=\overline{-n}\neq\overline{1-i}=\res(i+1,1+1)$ as $i$ and $n$ are odd and it follows that $i+1$ is $\lambda$-cogood.
	Note that $(\lambda+\varepsilon_{i+1})^\prime=\omega_{i+1}+\omega_n$, so by Theorem \ref{thm:BrundanKleshchevTensorNatural}, there is an embedding of $L(\omega_{i+1}+\omega_n)$ into $V\otimes L(\omega_i+\omega_n)$ and it follows that $U\coloneqq V\otimes L(\omega_{i+1}+\omega_n)$ can be embedded into $V\otimes V\otimes L(\omega_i+\omega_n)=M$.
	Now $U$ is not completely reducible by Proposition~\ref{prop:BrundanKleshchevGood} as $n$ is odd, so $U$ has Loewy length at least~$3$ by Corollary~\ref{cor:LoewyLengthAnp2} and we conclude that both $U$ and $M$ have Loewy length~$3$. Then $(\rad_G U+\soc_G U)/\soc_G U$ is non trivial and embeds into
	\[ (\rad_G M+\soc_G M)/\soc_G M\cong L(2\omega_1+\omega_i+\omega_n) . \] However, $U=V\otimes L(\omega_{i+1}+\omega_n)$ has highest weight $\omega_1+\omega_{i+1}+\omega_n<2\omega_1+\omega_i+\omega_n$, a contradiction.
\end{proof}

\begin{Proposition} \label{prop:Anp2}
	Let $\lambda,\mu\in X_1\setminus\{0\}$ such that $L(\lambda)\otimes L(\mu)$ is completely reducible. Then, up to reordering of $\lambda$ and $\mu$, we have $\lambda=\omega_i$ for some $1\leq i\leq n$ and $\mu=\omega_{j_1}+\cdots+\omega_{j_r}$ for some $1\leq j_1 < \cdots < j_r\leq n$ such that either $i<j_1$ or $i>j_r$.
\end{Proposition}
\begin{proof}
	As $\lambda,\mu\in X_1$ and $p=2$, we have $\lambda=\omega_{i_1}+\cdots+\omega_{i_s}$ for certain $1\leq i_1 < \cdots < i_s\leq n$ and $\mu=\omega_{j_1}+\cdots+\omega_{j_r}$ for certain $1\leq j_1 < \cdots < j_r\leq n$. By Corollary \ref{cor:highestweightvectorprestricted}, we have $\lambda+\mu\in X_1$ and it follows that the sets $I=\{i_1,\ldots,i_s\}$ and $J=\{j_1,\ldots, j_r\}$ are disjoint.  Let $1\leq x_1<\cdots < x_{r+s}\leq n$ such that $I\cup J=\{x_1,\ldots,x_{r+s}\}$ and let $1\leq t< r+s$ such that $x_t$ and $x_{t+1}$ do not belong to the same index set. We show that~$t\in\{1,r+s-1\}$.
	
	Suppose for a contradiction that $1<t$ and $t+1<r+s$. Without loss of generality, assume that $x_t\in I$ and $x_{t+1}\in J$. By truncating to $L_{[x_{t-1},x_{t+2}]}$, we may further assume that
	\[r+s=4, \qquad t=2, \qquad x_1=1, \qquad \text{and} \qquad x_4=n .\]
	If $1\in J$ or $n\in I$ then $L(\lambda)\otimes L(\mu)$ is not completely reducible by Lemma \ref{lem:Anp2highestroot}. Hence $1\in I$ and $n\in J$, so that $\lambda=\omega_1+\omega_{x_2}$ and $\mu=\omega_{x_3}+\omega_n$. Truncating to $L_{[x_2,n]}$ and applying Proposition~\ref{prop:BrundanKleshchevGood}, we see that $x_3-x_2+1$ and $n-x_2+1$ are even, in particular $n\neq x_3+1$. Then Lemma \ref{lem:Anp2omega2} shows that $L(\lambda)\otimes L(\mu)$ is not completely reducible, after truncating to $L_{[1,x_{3}+1]}$ and taking duals, a contradiction.
	
	Hence $t\in\{1,r+s-1\}$ and up to reordering of $\lambda$ and $\mu$, we have $I\subseteq \{x_1,x_{r+s}\}$. If $I=\{x_1,x_{r+s}\}$ then $\lambda=\omega_{x_1}+\omega_{x_{r+s}}$ and $L(\lambda)\otimes L(\mu)$ is not completely reducible by Lemma \ref{lem:Anp2highestroot} and a truncation argument, a contradiction. We conclude that $\lambda=\omega_{x_1}$ or $\lambda=\omega_{x_{r+s}}$, as required. 
\end{proof}

\begin{Remark} \label{rem:TensorFundamentalGF}
	Suppose that $n\geq 2$ and $i\in\{1,\ldots, n\}$ such that $i<n+1-i$. Then $L(\omega_i)^*\cong L(\omega_{n+1-i})$ and it follows that there is a canonical embedding $k\hookrightarrow L(\omega_i)\otimes L(\omega_{n+1-i})$. By applying the same argument to the truncation of $L(\omega_i)\otimes L(\omega_{n+1-i})$ to $L_{[j+1,n-j]}$ for $1\leq j< i$, we see that $L(\omega_{j}+\omega_{n+1-j})$ is a composition factor of $L(\omega_i)\otimes L(\omega_{n+1-i})$.
	
	Note that $L(\omega_i)=\nabla(\omega_i)$ and $L(\omega_{n+1-i})\cong \nabla(\omega_{n+1-i})$, so $L(\omega_i)\otimes L(\omega_{n+1-i})$ has a good filtration. Arguing as in the previous paragraph, we see that the sections of such a filtration are precisely the induced modules $\nabla(\omega_j+\omega_{n+1-j})$ for $1\leq j\leq i$ and $\nabla(0)$, each with multiplicity $1$.
\end{Remark}

\begin{Lemma} \label{lem:Anp2WeylHighestRoot}
	Suppose that $n\geq 2$. Then $\nabla(\omega_1+\omega_n)$ is irreducible if and only if $n$ is even.
\end{Lemma}
\begin{proof}
	By Remark \ref{rem:TensorFundamentalGF}, $L(\omega_1)\otimes L(\omega_n)$ has a good filtration with sections $\nabla(\omega_1+\omega_n)$ and $\nabla(0)=k$. If~$n$ is even then $L(\omega_1)\otimes L(\omega_n)$ is completely reducible by Proposition \ref{prop:BrundanKleshchevGood} and it follows that $\nabla(\omega_1+\omega_n)$ is simple. If $\nabla(\omega_1+\omega_n)$ is simple then $L(\omega_1)\otimes L(\omega_n)$ is multiplicity free, hence $L(\omega_1)\otimes L(\omega_n)$ is completely reducible by Lemma \ref{lem:MultiplicityFree} and $n$ is even by Proposition \ref{prop:BrundanKleshchevGood}.
\end{proof}

\begin{Lemma} \label{lem:Anp2Weylomega2omegan1}
	Suppose that $n\geq 4$. Then $\nabla(\omega_2+\omega_{n-1})$ is irreducible if and only if $n\equiv 2\mod 4$.
\end{Lemma}
\begin{proof}
	Suppose that $\nabla(\omega_2+\omega_{n-1})$ is irreducible. Then so is the truncation of $\nabla(\omega_2+\omega_{n-1})$ to $L_{[2,n-1]}$ and it follows that $n$ is even by Lemma \ref{lem:Anp2WeylHighestRoot}. By Remark \ref{rem:TensorFundamentalGF}, the tensor product $L(\omega_2)\otimes L(\omega_{n-1})$ has a good filtration with sections $\nabla(\omega_2+\omega_{n-1})$, $\nabla(\omega_1+\omega_n)$ and $\nabla(0)=k$. Furthermore, we have $L(\omega_{n-1})\cong L(\omega_2)^*$ and $\dim L(\omega_2)=\frac{n(n+1)}{2}$. If $n$ is divisible by $4$ then $\dim L(\omega_2)$ is even and by Remark~\ref{rem:Serre}, the canonical embedding $k\hookrightarrow L(\omega_2)\otimes L(\omega_2)^*$ does not split. Using contravariant duality, it follows that $k$ appears with multiplicity at least $2$ as a composition factor of $L(\omega_2)\otimes L(\omega_{n-1})$, hence~$k$ appears as a composition factor of one of the induced modules $\nabla(\omega_2+\omega_{n-1})$ and~$\nabla(\omega_1+\omega_n)$. Now $\nabla(\omega_1+\omega_n)$ is simple by Lemma \ref{lem:Anp2WeylHighestRoot} and it follows that $\nabla(\omega_2+\omega_{n-1})$ is non-simple.
	
	Now suppose that $n\equiv 2\mod 4$. Then $\dim L(\omega_2)$ is odd, so the canonical embedding
	\[k\hookrightarrow L(\omega_2)\otimes L(\omega_2)^*\]
	splits and we can write
	$L(\omega_2)\otimes L(\omega_{n-1}) \cong k \oplus M$
	for a $G$-module $M$ with a good filtration $0\leq N\leq M$ such that \[M/N\cong \nabla(\omega_2+\omega_{n-1})\qquad \text{and}\qquad N\cong \nabla(\omega_1+\omega_n) . \]
	Note that the only dominant weights below $\omega_2+\omega_n$ are $\omega_1+\omega_n$ and $0$. Now $L(\omega_1+\omega_n)$ is not a composition factor of $\nabla(\omega_2+\omega_{n-1})$ as the truncation of $\nabla(\omega_2+\omega_{n-1})$ to $L_{[2,n-1]}$ is simple by Lemma~\ref{lem:Anp2WeylHighestRoot}. Hence, if $\nabla(\omega_2+\omega_{n-1})$ is non-simple then $L(0)=k$ appears in the head of $\nabla(\omega_2+\omega_{n-1})$, hence in the head of $M$. Then $k$ appears with multiplicity $2$ in the head of $L(\omega_2)\otimes L(\omega_{n-1})$, contradicting Schur's lemma. Hence $\nabla(\omega_2+\omega_{n-2})$ is simple.
\end{proof}

\begin{Lemma} \label{lem:Anp2fundamental}
	Let $1\leq i < j\leq n$. Then $L(\omega_i)\otimes L(\omega_j)$ is completely reducible if and only if $j-i$ is odd and one of the following holds:
	\begin{enumerate}
	 \item $i=1$ or $j=n$,
	 \item $i=2$ or $j=n-1$ and $j-i\equiv 3\mod 4$.
	\end{enumerate}
\end{Lemma}
\begin{proof}
	If $j-i$ is odd and $i=1$ or $j=n$ then $L(\omega_i)\otimes L(\omega_j)$ is completely reducible by Proposition \ref{prop:BrundanKleshchevGood}. Now suppose that $i=2$, $j<n$ and $j-2\equiv 3\mod 4$.
	We have $L(\omega_2)=\nabla(\omega_2)$ and $L(\omega_j)=\nabla(\omega_j)$, so $L(\omega_2)\otimes L(\omega_j)$ has a good filtration with sections $\nabla(\omega_2+\omega_j)$, $\nabla(\omega_1+\omega_{j+1})$ and $\nabla(\omega_{j+2})$, where we consider $\omega_{j+2}=0$ if $j+2>n$. Indeed, this is clear for $j=n-1$ by Remark \ref{rem:TensorFundamentalGF}, in the general case it follows by truncating to $L_{[1,j+1]}$ as $\omega_1+\omega_{j+1}$ and $\omega_{j+2}$ are the only dominant weights below~$\omega_2+\omega_j$. Using a similar truncation argument, it follows from Lemmas \ref{lem:Anp2WeylHighestRoot} and \ref{lem:Anp2Weylomega2omegan1} that the modules $\nabla(\omega_2+\omega_j)$ and $\nabla(\omega_1+\omega_{j+1})$ are simple. Thus the good filtration of $L(\omega_2)\otimes L(\omega_j)$ is also a composition series and $L(\omega_2)\otimes L(\omega_j)$ is multiplicity free, hence completely reducible by Lemma \ref{lem:MultiplicityFree}. If $j=n-1$, $i>1$ and $j-i\equiv 3\mod 4$ then the claim follows as before by dualizing.
	
	Now let $i$ and $j$ be arbitrary and suppose that $L(\omega_i)\otimes L(\omega_j)$ is completely reducible. The induced module $\nabla(\omega_i+\omega_j)$ appears in a good filtration of $L(\omega_i)\otimes L(\omega_j)$, hence $\nabla(\omega_i+\omega_j)$ is simple. By truncating to $L_{[i,j]}$ and applying Lemma \ref{lem:Anp2WeylHighestRoot}, we see that $j-i$ is odd. Furthermore, if $i>1$ and $j<n$ then truncating to $L_{[i-1,j+1]}$ and applying Lemma \ref{lem:Anp2Weylomega2omegan1} yields $j-i\equiv 3\mod 4$. Finally, suppose that $i>2$ and $j<n-1$, we show that $L(\omega_i)\otimes L(\omega_j)$ is not completely reducible. By truncting to $L_{[i-2,j+2]}$, we may assume that $i=3$ and $j=n-2$. The induced modules $\nabla(\omega_3+\omega_{n-2})$ and $\nabla(\omega_2+\omega_{n-1})$ appear in a good filtration of $L(\omega_3)\otimes L(\omega_{n-2})$. Suppose for a contradiction that $L(\omega_3)\otimes L(\omega_{n-2})$ is completely reducible. Then $\nabla(\omega_2+\omega_{n-1})$ is simple and $n\equiv 2\mod 4$ by Lemma \ref{lem:Anp2Weylomega2omegan1}. Analogously, the truncation of $\nabla(\omega_3+\omega_{n-2})$ to $L_{[2,n-1]}$ is simple and Lemma \ref{lem:Anp2Weylomega2omegan1} yields $n-2\equiv 2\mod 4$, a contradiction.
\end{proof}

\begin{Theorem} \label{thm:ClassificationAnp2}
	Let $G$ be of type $\mathrm{A}_n$ and $p=2$. 
	Let $\lambda,\mu\in X_1\setminus\{0\}$. Then $L(\lambda)\otimes L(\mu)$ is completely reducible if and only if one of the following holds, up to reordering $\lambda$ and $\mu$:
	\begin{enumerate}
	\item $\lambda=\omega_1$ and $\mu=\omega_{i_1}+\cdots+\omega_{i_r}$ for even numbers $i_1,\ldots,i_r$ with $1<i_1<\ldots<i_r\leq n$,
	\item $\lambda=\omega_n$ and $\mu=\omega_{i_1}+\cdots+\omega_{i_r}$ for certain $1\leq i_1<\cdots <i_r< n$ such that $n+1-i_j$ is even for all $j\in\{1,\ldots, r\}$,
	\item $\lambda=\omega_2$ and $\mu=\omega_j$ for some $2<j\leq n$ with $j-2\equiv 3\mod 4$,
	\item $\lambda=\omega_{n-1}$ and $\mu=\omega_j$ for some $1\leq j<n-1$ with $n-1-j\equiv 3\mod 4$.
	\end{enumerate}
\end{Theorem}
\begin{proof}
	If $\lambda$ and $\mu$ are as in points (1) to (4) above then $L(\lambda)\otimes L(\mu)$ is completely reducible by Proposition \ref{prop:BrundanKleshchevGood} and Lemma \ref{lem:Anp2fundamental}.
	Now suppose that $L(\lambda)\otimes L(\mu)$ is completely reducible. By Proposition~\ref{prop:Anp2}, we may assume that $\lambda=\omega_i$ for some $1\leq i\leq n$ and $\mu=\omega_{i_1}+\cdots+\omega_{i_r}$ with $1\leq i_1 < \cdots < i_r\leq n$ such that either $i<i_1$ or $i_r<i$. If $\lambda=\omega_1$ then $\mu$ is as in (1) and if $\lambda=\omega_n$ then $\mu$ is as in (2) by Proposition \ref{prop:BrundanKleshchevGood}. Now assume that $1<i<n$. If $r\geq 2$ then $L(\lambda)\otimes L(\mu)$ is not completely reducible by Lemma \ref{lem:Anp2omega2} and a truncation argument, a contradiction. Hence $r=1$ and $\mu=\omega_{i_1}$. Now the claim follows from Lemma \ref{lem:Anp2fundamental}.
\end{proof}

\section{Results for small primes} \label{sec:smallprimes}

In this section, we give proofs of Theorems \ref{thm:introA} and \ref{thm:introB} for $G$ of type $\mathrm{B}_n$, $\mathrm{C}_n$ and $\mathrm{F}_4$ when $p=2$, and for $G$ of type $\mathrm{G}_2$ when $p=3$.

Suppose that $\Phi$ is of type $\mathrm{B}_n$, $\mathrm{C}_n$, $\mathrm{F}_4$ or $\mathrm{G}_2$ and denote by $\Phi_\ell$ and $\Phi_s$ the long roots and the short roots in $\Phi$, respectively. Furthermore, define $\Delta_\ell=\Phi_\ell\cap \Delta$ and $\Delta_s=\Phi_s\cap\Delta$ and let
\[ X_\ell=\{\lambda\in X\mid\langle\lambda,\alpha^\vee\rangle=0\text{ for all }\alpha\in\Delta_s\} \quad\text{and}\quad X_s=\{\lambda\in X\mid\langle\lambda,\alpha^\vee\rangle=0\text{ for all }\alpha\in\Delta_\ell\} , \]
so that $X=X_\ell\oplus X_s$.
The following theorem due to R. Steinberg provides a refinement of the tensor product theorem in small characteristic, see Theorem 11.1 in \cite{Steinberg}:

\begin{Theorem}[Steinberg] \label{thm:SteinbergIsogeny}
	Suppose that $p=2$ and $G$ is of type $\mathrm{B}_n$, $\mathrm{C}_n$ or $\mathrm{F}_4$ or that $p=3$ and $G$ is of type $\mathrm{G}_2$. Let $\lambda\in X^+$ and write $\lambda=\lambda_\ell+\lambda_s$ with $\lambda_\ell\in X_\ell$ and $\lambda_s\in X_s$. Then $L(\lambda)\cong L(\lambda_\ell)\otimes L(\lambda_s)$.
\end{Theorem}

\subsection{Type \texorpdfstring{$\mathrm{B}_n$}{Bn}}

Let $G$ be of type $\mathrm{B}_n$ and $p=2$. The labeling of simple roots is such that $\Delta_s=\{\alpha_n\}$ and $\Delta_\ell=\{\alpha_1,\ldots,\alpha_{n-1}\}$.

\begin{Lemma} \label{lem:MaxVecLastWeightBn}
	Let $V$ be a rational $G$-module and suppose that $V$ has a maximal vector $v$ of weight $\delta=\delta_0+2\omega_n$, with $\delta_0\in X_1$. If $V$ is completely reducible as a $G_1$-module then $v$ generates a simple $G$-submodule of $V$.
\end{Lemma}
\begin{proof}
	Suppose that $V$ is completely reducible as a $G_1$-module. Then $U_k(\mathfrak{g})\cdot v$ is a $G$-submodule of
	\[ \head_{G_1}\Delta(\delta)\cong L(\delta_0)\otimes\Delta(\omega_n)^{[1]} \cong L(\delta_0)\otimes L(\omega_n)^{[1]} \cong L(\delta_0+2\omega_n) , \]
	where the first isomorphism follows from Corollary \ref{cor:headG1} and $\Delta(\omega_n)=L(\omega_n)$ as $\omega_n$ is a minuscule weight. Hence $v$ generates a simple $G$-submodule of $V$.
\end{proof}

\begin{Proposition} \label{prop:Bn}
	Let $\lambda,\mu\in X_1$. Then $L(\lambda)\otimes L(\mu)$ is completely reducible if and only if $L(\lambda)\otimes L(\mu)$ is completely reducible as a $G_1$-module. Moreover, if $L(\lambda)\otimes L(\mu)$ is completely reducible then all composition factors of $L(\lambda)\otimes L(\mu)$ are $p$-restricted.
\end{Proposition}
\begin{proof}
	As before, if $L(\lambda)\otimes L(\mu)$ is completely reducible as a $G$-module then $L(\lambda)\otimes L(\mu)$ is completely reducible as a $G_1$-module.
	Suppose now that $L(\lambda)\otimes L(\mu)$ is completely reducible as a $G_1$-module, in particular $\lambda+\mu\in X_1$ by Corollary \ref{cor:highestweightvectorprestricted}. We can write
	\[\lambda=\lambda^\prime + \varepsilon_\lambda \omega_n\qquad\text{and}\qquad\mu=\mu^\prime + \varepsilon_\mu \omega_n\]
	with $\lambda^\prime,\mu^\prime\in X_\ell$ and $\varepsilon_\lambda,\varepsilon_\mu\in\{0,1\}$ not both equal to $1$, so that
	\[ L(\lambda)\cong L(\lambda^\prime)\otimes L(\varepsilon_\lambda \omega_n) \qquad \text{and} \qquad L(\mu)\cong L(\mu^\prime)\otimes L(\varepsilon_\mu \omega_n) \]
	by Theorem \ref{thm:SteinbergIsogeny}. If $\lambda^\prime=0$ or $\mu^\prime=0$ then Theorem \ref{thm:SteinbergIsogeny} implies that $L(\lambda)\otimes L(\mu)$ is simple of highest weight $\lambda+\mu\in X^+$, so now assume that $\lambda^\prime,\mu^\prime\neq 0$.
	
	The truncation of $L(\lambda)\otimes L(\mu)$ to $L_{[1,n-1]}$ is completely reducible when restricted to the Frobenius kernel of $L_{[1,n-1]}$ and by Theorems \ref{thm:completelyreducibleGandG1} and \ref{thm:ClassificationAnp2}, we may assume that $\lambda^\prime=\omega_i$ with $i\in\{1,2,n-2,n-1\}$. We write $\mu^\prime=\omega_{i_1}+\cdots+\omega_{i_r}$ with indices $1\leq i_1<\cdots<i_r\leq n-1$ and consider the four possibilities in turn:
	\begin{enumerate}
		\item Suppose that $\lambda^\prime=\omega_1$. We have
		\[ L(\lambda)\otimes L(\mu)\cong L(\lambda^\prime)\otimes L(\varepsilon_\lambda \omega_n) \otimes L(\mu) \cong L(\lambda^\prime)\otimes L(\mu+\varepsilon_\lambda\omega_n) \]
		by Theorem \ref{thm:SteinbergIsogeny} as $\varepsilon_\lambda+\varepsilon_\mu\leq 1$. After replacing $\mu$ by $\mu+\varepsilon_\lambda\omega_n$, we may assume that $\lambda=\lambda^\prime=\omega_1$. Then $\langle\lambda,\alpha_0^\vee\rangle=1<2$ and $\soc_G\big( L(\lambda)\otimes L(\mu) \big)$ is $p$-restricted by Theorem \ref{thm:BrundanKleshchevsocle}.
		
		If all maximal vectors in $L(\lambda)\otimes L(\mu)$ have $p$-restricted weight then the claim follows from Proposition \ref{prop:GandG1allMaxVecPRes}, so now assume for a contradiction that $L(\lambda)\otimes L(\mu)$ has a maximal vector $v$ of weight $\delta\in X^+\setminus X_1$. For $j<n$, we have $-\langle\beta,\alpha_j^\vee\rangle<2$ for all $\beta\in\Phi^+$ as $\alpha_j\in\Delta_\ell$, so $\langle\delta,\alpha_j^\vee\rangle<2$ by Proposition \ref{prop:nonpresnonsimpleG1} and it follows that $\langle\delta,\alpha_n^\vee\rangle\geq 2$. Furthermore, $\delta-\mu$ is a weight of $L(\lambda)=L(\omega_1)$ by Lemma \ref{lem:weaklymaximaltensor} and it follows that $\langle\delta,\alpha_n^\vee\rangle\in\{2,3\}$ as $\langle\nu,\alpha_n^\vee\rangle\leq 2$ for all $\nu\in X$ with $L(\omega_1)_\nu\neq 0$. As $L(\lambda)\otimes L(\mu)$ has $p$-restricted socle, $v$ generates a non-simple $G$-submodule of $L(\lambda)\otimes L(\mu)$ and by Lemma \ref{lem:MaxVecLastWeightBn}, $L(\lambda)\otimes L(\mu)$ is not completely reducible as a $G_1$-module, a contradiction.
		\item Suppose that $\lambda^\prime=\omega_2$ and $\mu^\prime\neq\omega_1$ so that $r=1$ by Theorem \ref{thm:ClassificationAnp2}, and let $2<j=i_1$.
		If $j=n-1$ then $L(\lambda)\otimes L(\mu)$ is not completely reducible as a $G_1$-module. Indeed, by truncating to $L_{[2,n]}$, we may assume that $\lambda^\prime=\omega_1$ so that $L(\lambda)\otimes L(\mu)\cong L(\omega_1)\otimes L(\omega_{n-1})\otimes L((\varepsilon_\lambda+\varepsilon_\mu)\cdot\omega_n)$ has a composition factor of non-$2$-restricted highest weight $(2+\varepsilon_\lambda+\varepsilon_\mu)\cdot\omega_n$ and the claim follows from part (1).
		If $j=n-2$ then by truncating to $L_{[1,n-1]}$ and arguing as in Remark~\ref{rem:TensorFundamentalGF}, we see that $L(\lambda)\otimes L(\mu)$ has a composition factor of weight $(2+\varepsilon_\lambda+\varepsilon_\mu)\cdot \omega_n$. If $\varepsilon_\lambda+\varepsilon_\mu=0$ then $L(\lambda)\otimes L(\mu)$ is not completely reducible by Lemma \ref{lem:DonkinNonCR}. So $\varepsilon_\lambda+\varepsilon_\mu=1$ and
		\[ L(\lambda)\otimes L(\mu)\cong L(\omega_2)\otimes L(\omega_{n-2}) \otimes L(\omega_n) . \]
	Moreover, we have $L(3\omega_n)\cong L(\omega_n)\otimes L(\omega_n)^{[1]}$, where $L(\omega_n)^{[1]}$ is a trivial $G_1$-module and as $L(\lambda)\otimes L(\mu)$ is completely reducible as a $G_1$-module, it follows that
		\begin{align*}
			0 &\neq \Hom_{G_1}\big(L(\omega_n),L(\omega_2)\otimes L(\omega_{n-2})\otimes L(\omega_n)\big) \\
			& \cong \Hom_{G_1}\big(L(\omega_2)\otimes L(\omega_n),L(\omega_{n-2})\otimes L(\omega_n)\big) \\
			& \cong \Hom_{G_1}\big(L(\omega_2+\omega_n),L(\omega_{n-2}+\omega_n)\big)
		\end{align*}
		by Theorem \ref{thm:SteinbergIsogeny}. This implies that $2=n-2$ and $\lambda+\mu$ is non-$2$-restricted, a contradiction.
		Finally, if $j<n-2$ then all maximal vectors in $L(\omega_2)\otimes L(\omega_i)$ have $2$-restricted weight as all dominant weights below $\omega_2+\omega_j$ are $2$-restricted, and the claim follows from Proposition \ref{prop:GandG1allMaxVecPRes}.
		\item Suppose that $\lambda^\prime=\omega_{n-1}$ and by truncating to $L_{[i_r,n]}$, assume that $\mu^\prime=\omega_1$. Then by truncating to $L_{[1,n-1]}$, we see that $L(\lambda)\otimes L(\mu)$ has a composition factor of non-$2$-restricted highest weight
		\[ \lambda+\mu-\alpha_1-\ldots-\alpha_{n-1}=(2+\varepsilon_\lambda+\varepsilon_\mu) \cdot \omega_n , \]
		contradicting the assumption that $L(\lambda)\otimes L(\mu)$ is completely reducible, by part (1).
		\item Suppose that $\lambda^\prime=\omega_{n-2}$ and $\mu^\prime\notin\{\omega_1,\omega_{n-1}\}$, so $r=1$ by Theorem \ref{thm:ClassificationAnp2} and we let $j=i_1>1$. By truncating to $L_{[j-1,n]}$, we may assume that $\mu^\prime=\omega_2$. Then by truncating to~$L_{[1,n-1]}$ and arguing as in Remark \ref{rem:TensorFundamentalGF}, we see that $L(\lambda)\otimes L(\mu)$ has a composition factor of non-$2$-restricted highest weight $(2+\varepsilon_\lambda+\varepsilon_\mu)\cdot \omega_n$
		contradicting the assumption that $L(\lambda)\otimes L(\mu)$ is completely reducible, as in part (2). \qedhere
	\end{enumerate}
\end{proof}

\begin{Remark}
	The preceding proposition shows that Theorems \ref{thm:introA} and \ref{thm:introB} are valid for $G$ of type $\mathrm{B}_n$ when $p=2$.
\end{Remark}

\subsection{Type \texorpdfstring{$\mathrm{C}_n$}{Cn}}

Let $G$ be of type $\mathrm{C}_n$ and $p=2$. We start with an easy lemma about the root system of $G$.

\begin{Lemma} \label{lem:Cnminus2}
	Let $\beta\in\Phi^+$ and $\alpha\in\Delta$. If $\langle\beta,\alpha^\vee\rangle=-2$ then there exists $\gamma\in\Delta$ such that $\langle\beta,\gamma^\vee\rangle=2$.
\end{Lemma}
\begin{proof}
	Consider the real vector space $E=\R^n$ with canonical basis $e_1,\ldots,e_n$, equipped with the standard scalar product $\langle\cdot\,,\cdot\rangle$. The root system of type $\mathrm{C}_n$ can be realized in $E$ as $\Phi=\Phi^+\sqcup\Phi^-$ with 
	\[\Phi^+=\{ e_i\pm e_j \mid 1\leq i<j\leq n\}\cup\{ 2 e_i \mid 1\leq i\leq n \} \]
	and the simple roots are $\Delta=\{e_i-e_{i+1}\mid 1\leq i<n \}\cup \{2e_n\}$.
	If $\langle\beta,\alpha^\vee\rangle=-2$ then $\beta$ is long and $\alpha$ is short, so $\beta=2e_i$ for some $i\in\{1,\ldots,n\}$. If $i=n$, we can take $\gamma=\beta=2e_n\in\Delta$. If $i<n$ then the claim follows with $\gamma\coloneqq e_i-e_{i+1}\in \Delta$.
\end{proof}

The following result is a weaker version of Proposition \ref{prop:weakmaxvecnotG2}, the proof is almost the same.

\begin{Lemma} \label{lem:weakmaxvecCnp2}
	Let $\lambda,\mu\in X_1$ and let $v\in L(\lambda)\otimes L(\mu)$ be a weakly maximal vector of weight $\delta\in X^+ \setminus X_1$. Assume furthermore that there exists $\beta\in\Phi^+$ such that $w\coloneqq (1\otimes X_\beta)\cdot v$ is a weakly maximal vector in $L(\lambda)\otimes L(\mu)$ and $\delta+\beta\in X_1^\prime$. Then $v$ generates a non-simple $G_1$-submodule of $L(\lambda)\otimes L(\mu)$.
\end{Lemma}
\begin{proof}
	Let $\alpha\in\Delta$ such that $\langle\delta,\alpha^\vee\rangle\geq p$. We have $\langle\delta+\beta,\alpha^\vee\rangle < p$ as $\delta+\beta\in X_1^\prime$, so $\langle\beta,\alpha^\vee\rangle<0$ and therefore $\beta-\alpha\notin\Phi$ and $\langle\beta,\alpha^\vee\rangle\in\{-1,-2\}$ by Lemma \ref{lem:rootnegativecoeff}. As in the proof of Proposition \ref{prop:weakmaxvecnotG2}, it follows that $X_{-\alpha}\cdot v\neq 0$ whenever $X_{-\alpha}\cdot w\neq 0$.
	
	If $\langle\beta,\alpha^\vee\rangle=-2$ then there exists $\gamma\in\Delta$ such that $\langle\beta,\gamma^\vee\rangle=2$ by Lemma \ref{lem:Cnminus2}. Then $\langle\delta+\beta,\gamma^\vee\rangle\geq 2$ and $\delta+\beta\notin X_1^\prime$, a contradiction. Hence $\langle\beta,\alpha^\vee\rangle=-1$ and it follows that $\langle\delta,\alpha^\vee\rangle=2$ and $\langle\delta+\beta,\alpha^\vee\rangle=1$. As $w$ is a weakly maximal vector of weight $\delta+\beta$, we have $X_{-\alpha}w\neq 0$ and hence $X_{-\alpha}v\neq 0$. Now the claim follows from Lemma \ref{lem:weakmaxvecnonsimpleG1}.
\end{proof}

\begin{Proposition} \label{prop:Cnp2nonCRG1}
	Let $\lambda,\mu\in X_1$ and suppose that $L(\lambda)\otimes L(\mu)$ has a weakly maximal vector $v$ of weight $\delta\in X\setminus X_1^\prime$. Then $L(\lambda)\otimes L(\mu)$ is not completely reducible as a $G_1$-module.
\end{Proposition}
\begin{proof}
	Suppose for a contradiction that $V\coloneqq L(\lambda)\otimes L(\mu)$ is completely reducible as a $G_1$-module, so~$\lambda+\mu\in X_1$ by Corollary \ref{cor:highestweightvectorprestricted}. Recall that $V$ is a submodule of $\nabla(\lambda)\otimes\nabla(\mu)$ and $\nabla(\lambda)\otimes\nabla(\mu)$ has a good filtration. Fix a good filtration
	\[0=M_0\leq\cdots\leq M_s=\nabla(\lambda)\otimes\nabla(\mu)\]
	with $M_i/M_{i-1}\cong\nabla(\lambda_i)$ for $i>0$ such that $\lambda_i>\lambda_j$ implies $i>j$, and set $N_i\coloneqq M_i\cap V$. Then
	\[0=N_0\leq\cdots\leq N_s=V\]
	is a filtration of $V$ such that $N_i/N_{i-1}$ is a (possibly zero) submodule of $\nabla(\lambda_i)$. Starting from $v_0\coloneqq v$ and $\delta_0\coloneqq \delta$, we construct a sequence of weakly maximal vectors $v_i\in V$ of weight $\delta_i\in X\setminus X_1^\prime$ such that $\delta_i<\delta_{i+1}$ for all $i\geq 0$ as follows:
	
	Suppose that $v_i$ of weight $\delta_i$ has been constructed and let $j\in\{1,\ldots,s\}$ such that $v_i\in N_j\setminus N_{j-1}$. Then the image of $u_k(\mathfrak{g})\cdot v_i$ in $N_j/N_{j-1}$ is non-zero and lies in the $G_1$-socle of $N_j/N_{j-1}$, hence in the $G_1$-socle of $\nabla(\lambda_j)$. If $\lambda_j\in X_1$ then $\soc_{G_1}\nabla(\lambda_j)\cong L(\lambda_j)$ by Lemma \ref{lem:weylgeneratedG1}, so $\delta_i=\lambda_j$ contradicting the assumption that $\delta_i\in X\setminus X_1^\prime$. Hence $\lambda_j$ is non-$p$-restricted, in particular $\lambda_j\neq\lambda+\mu$. Moreover, by Corollary 4.6 in \cite{BrundanKleshchevCR}, we have
	\[ 0\neq \Hom_G(\Delta(\lambda_j),V) , \]
	so $V$ has a maximal vector $w$ of weight $\lambda_j$. By Proposition \ref{prop:nothighestweighthigherweakmaxvec}, there exists $\beta\in\Phi^+$ such that $(1\otimes X_\beta)\cdot w$ is a weakly maximal vector. If $\lambda_j+\beta\in X_1^\prime$ then $w$ generates a non-simple $G_1$-submodule of $V$ by Lemma \ref{lem:weakmaxvecCnp2}, a contradiction. Hence $\lambda_j+\beta\in X\setminus X_1^\prime$ and we set $v_{i+1}=(1\otimes X_\beta)\cdot w$, a weakly maximal vector of weight $\delta_{i+1}=\lambda_j+\beta>\delta_i$.
	
	Thus, we obtain a sequence of weakly maximal vectors $v_0,v_1,\ldots$ in $V$ of weights $\delta_0<\delta_1<\cdots$. However, the set of weights of $V$ is finite, a contradiction.
\end{proof}

\begin{Remark} \label{rem:Cnp2thmAB}
	The preceding proposition shows that Corollary \ref{cor:nonpresnonsimpleG1} is also valid for $G$ of type $\mathrm{C}_n$ when $p=2$. Recall that Proposition \ref{prop:GandG1allMaxVecPRes} is valid in arbitrary characteristic. Now Theorems \ref{thm:introA} and \ref{thm:introB} can be proven exactly as in Section \ref{sec:CRprestricted}, using Proposition \ref{prop:Cnp2nonCRG1} instead of Corollary \ref{cor:nonpresnonsimpleG1}, see Remark \ref{rem:afterABforCnp2}.
\end{Remark}

The following lemma will be used in the next subsection, where we consider the case of $\mathrm{F}_4$.

\begin{Lemma} \label{lem:C3p2}
	Assume that $\Phi$ is of type $\mathrm{C}_3$ and $p=2$ and let $\lambda,\mu\in X_1\setminus \{0\}$. Then $L(\lambda)\otimes L(\mu)$ is completely reducible as a $G_1$-module if and only if, up to reordering of $\lambda$ and $\mu$, $\lambda=\lambda_\ell$ and $\mu=\mu_s$. In that case, $L(\lambda)\otimes L(\mu)\cong L(\lambda+\mu)$.
\end{Lemma}
\begin{proof}
	The labeling of simple roots is such that $\Delta_s=\{\alpha_1,\alpha_2\}$ and $\Delta_\ell=\{\alpha_3\}$.

	If $\lambda=\lambda_\ell$ and $\mu=\mu_s$ then $L(\lambda)\otimes L(\mu)\cong L(\lambda+\mu)$ by Theorem \ref{thm:SteinbergIsogeny}, so $L(\lambda)\otimes L(\mu)$ is completely reducible as a $G_1$-module. Now suppose that $L(\lambda)\otimes L(\mu)$ is completely reducible as a $G_1$-module. Then $\lambda+\mu\in X_1$ by Corollary \ref{cor:highestweightvectorprestricted} and it follows that, up to reordering of $\lambda$ and $\mu$, we have $\lambda_\ell=0$ and so $\lambda_s\neq 0$. Suppose for a contradiction that $\mu_s\neq 0$. Then by Theorem \ref{thm:SteinbergIsogeny}, we have
	\[ L(\lambda)\otimes L(\mu) \cong L(\omega_1)\otimes L(\omega_2)\otimes L(\mu_\ell) , \]
	where either $\mu_\ell=0$ or $\mu_\ell=\omega_3$. If $\mu_\ell=\omega_3$ then $L(\lambda)\otimes L(\mu)$ has a composition factor of highest weight~$2\omega_3$, contradicting Proposition \ref{prop:Cnp2nonCRG1}. Furthermore, we can compute using GAP that $L(\omega_1)$ appears with multiplicity two as a composition factor of $L(\omega_1)\otimes L(\omega_2)$ and that
	\[ \dim \Hom_G\big(\Delta(\omega_1),\nabla(\omega_1)\otimes \nabla(\omega_2)\big)=1 , \]
	hence $L(\omega_1)\otimes L(\omega_2)$ is not completely reducible as a $G$-module and also not completely reducible as a $G_1$-module by Remark \ref{rem:Cnp2thmAB}, a contradiction.
\end{proof}

\subsection{Type \texorpdfstring{$\mathrm{F}_4$}{F4}}

Let $G$ be of type $\mathrm{F}_4$ and $p=2$. The labeling of simple roots is such that $\Delta_\ell=\{\alpha_1,\alpha_2\}$ and $\Delta_s=\{\alpha_3,\alpha_4\}$, hence $L_{[1,3]}$ is of type $\mathrm{B}_3$ and $L_{[2,4]}$ is of type $\mathrm{C}_3$.

\begin{Proposition} \label{prop:F4p2}
	Let $\lambda,\mu\in X_1\setminus\{0\}$. Then $L(\lambda)\otimes L(\mu)$ is completely reducible as a $G_1$-module if and only if, up to reordering of $\lambda$ and $\mu$, $\lambda=\lambda_\ell$ and $\mu=\mu_s$. In that case, $L(\lambda)\otimes L(\mu)\cong L(\lambda+\mu)$.
\end{Proposition}
\begin{proof}
	If $\lambda=\lambda_\ell$ and $\mu=\mu_s$ then $L(\lambda)\otimes L(\mu)\cong L(\lambda+\mu)$ by Theorem \ref{thm:SteinbergIsogeny}, so $L(\lambda)\otimes L(\mu)$ is completely reducible as a $G_1$-module. Now suppose that $L(\lambda)\otimes L(\mu)$ is completely reducible as a $G_1$-module, so~$\lambda+\mu\in X_1$ by Corollary \ref{cor:highestweightvectorprestricted}. By truncating to $L_{[2,4]}$ and applying Lemma \ref{lem:C3p2}, we obtain that, up to reordering of $\lambda$ and $\mu$, $\lambda_s=0$ and so $\lambda_\ell\neq 0$. Suppose for a contradiction that $\mu_\ell\neq 0$. Then by Corollary \ref{cor:highestweightvectorprestricted} and Theorem \ref{thm:SteinbergIsogeny}, we have
	\[ L(\lambda)\otimes L(\mu) \cong L(\omega_1)\otimes L(\omega_2) \otimes L(\mu_s) . \]
	It follows that the truncation of $L(\lambda)\otimes L(\mu)$ to $L_{[1,3]}$ has a composition factor of non-$2$-restricted highest weight $(\langle\mu_s,\alpha_3^\vee\rangle+2)\cdot\omega_3$, contradicting Theorem \ref{prop:Bn}.
\end{proof}

\begin{Remark} \label{rem:F4p2}
	Theorems \ref{thm:introA} and \ref{thm:introB} are now immediate from Proposition \ref{prop:F4p2}: If $\lambda,\mu\in X_1$ such that $L(\lambda)\otimes L(\mu)$ is completely reducible as a $G_1$-module then $\lambda+\mu\in X_1$ and $L(\lambda)\otimes L(\mu)\cong L(\lambda+\mu)$, in particular $L(\lambda)\otimes L(\mu)$ is completely reducible as a $G$-module and all composition factors are $p$-restricted.
\end{Remark}
	
\subsection{Type \texorpdfstring{$\mathrm{G}_2$}{G2}}
	
Let $G$ be of type $\mathrm{G}_2$ and $p=3$. We choose the same labeling of simple roots as in Section \ref{sec:G2}, that is $\Delta_s=\{\alpha_1\}$ and $\Delta_\ell=\{\alpha_2\}$. As is Section \ref{sec:G2}, we use GAP (and specifically \cite{Doty}) to compute the composition factors of Weyl modules and of tensor products of simple modules.

\begin{Proposition} \label{prop:G2p3}
	Let $\lambda,\mu\in X_1\setminus\{0\}$. Then $L(\lambda)\otimes L(\mu)$ is completely reducible as a $G_1$-module if and only if, up to reordering of $\lambda$ and $\mu$, $\lambda=\lambda_\ell$ and $\mu=\mu_s$. In that case, $L(\lambda)\otimes L(\mu)\cong L(\lambda+\mu)$.
\end{Proposition}
\begin{proof}
	If $\lambda=\lambda_\ell$ and $\mu=\mu_s$ then $L(\lambda)\otimes L(\mu)\cong L(\lambda+\mu)$ by Theorem \ref{thm:SteinbergIsogeny}, so $L(\lambda)\otimes L(\mu)$ is completely reducible as a $G_1$-module. We show that $L(\lambda)\otimes L(\mu)$ is not completely reducible as a $G_1$-module in the remaining cases. By Corollary \ref{cor:highestweightvectorprestricted}, it is sufficient to consider $\lambda,\mu\in X_1$ such that $\lambda+\mu\in X_1$.
	
	We can compute using GAP that that $L(3\omega_1)\cong L(\omega_1)^{[1]}$ is a composition factor of $L(\omega_2)\otimes L(\omega_2)$, so $L(\omega_2)\otimes L(\omega_2)$ is not completely reducible as a $G_1$-module by Lemma \ref{lem:DonkinNonCR}.
	As $L(\omega_1)$ is $7$-dimensional, we conclude from Proposition \ref{prop:Serre} and Remark \ref{rem:Serre} that the tensor products
	\[ L(\omega_1+\omega_2)\otimes L(\omega_2)\cong L(\omega_1)\otimes L(\omega_2)\otimes L(\omega_2) \]
	and
	\[ L(\omega_1+\omega_2)\otimes L(\omega_1+\omega_2)\cong L(\omega_1) \otimes L(\omega_1)\otimes L(\omega_2)\otimes L(\omega_2) \]
	are not completely reducible as $G_1$-modules.
	
	Furthermore, we have $L(\omega_1)=\nabla(\omega_1)$ and it follows that $L(\omega_1)\otimes L(\omega_1)=\nabla(\omega_1)\otimes\nabla(\omega_1)$ has a good filtration. We can compute that $\nabla(\omega_2)$ appears in a good filtration of $L(\omega_1)\otimes L(\omega_1)$, and that~$\nabla(\omega_2)$ is non-simple. Then $\nabla(\omega_2)$ is not completely reducible as a $G_1$-module by Lemma \ref{lem:weylgeneratedG1} and it follows that $L(\omega_1)\otimes L(\omega_1)$ is not completely reducible as a $G_1$-module. As $\dim L(\omega_2)=7$, we conclude from Proposition \ref{prop:Serre} and Remark \ref{rem:Serre} that the tensor product
	\[ L(\omega_1+\omega_2)\otimes L(\omega_1) \cong L(\omega_2)\otimes L(\omega_1)\otimes L(\omega_1) \]
	is not completely reducible as a $G_1$-module.
	
	It remains to consider the tensor products $L(2\omega_1+\omega_2)\otimes L(\omega_2)$ and $L(\omega_1+2\omega_2)\otimes L(\omega_1)$. We can compute that $L(3\omega_2)\cong L(\omega_2)^{[1]}$ is a composition factor of $L(\omega_1+2\omega_2)\otimes L(\omega_1)$, so $L(\omega_1+2\omega_2)\otimes L(\omega_1)$ is not completely reducible as a $G_1$-module by Lemma \ref{lem:DonkinNonCR}. We can further compute that $L(5\omega_1)$ is the unique non-$3$-restricted composition factor of $L(2\omega_1+\omega_2)\otimes L(\omega_2)$. By Theorem \ref{thm:BrundanKleshchevsocle}, the socle of $L(2\omega_1+\omega_2)\otimes L(\omega_2)$ is $p$-restricted, so $L(\lambda)\otimes L(\mu)$ is not completely reducible. If all maximal vectors in $L(2\omega_1+\omega_2)\otimes L(\omega_2)$ have $3$-restricted weight then $L(2\omega_1+\omega_2)\otimes L(\omega_2)$ is not completely reducible as a $G_1$-module by Proposition \ref{prop:GandG1allMaxVecPRes}. Now suppose for a contradiction that $L(2\omega_1+\omega_2)\otimes L(\omega_2)$ has a maximal vector $v$ of non-$p$-restricted weight $\delta$ and that $L(2\omega_1+\omega_2)\otimes L(\omega_2)$ is completely reducible as a $G_1$-module. Then $\delta=5\omega_1$ and $U\coloneqq U_k(\mathfrak{g})\cdot v$ is isomorphic to a quotient of
	\[\head_{G_1}\Delta(5\omega_1)\cong L(2\omega_1)\otimes\Delta(\omega_1)^{[1]} \cong L(5\omega_1) \]
	by Corollary \ref{cor:headG1} as $\Delta(\omega_1)=L(\omega_1)$. Thus $U\cong L(5\omega_1)$ is a non-$p$-restricted simple module in the socle of $L(2\omega_1+\omega_2)\otimes L(\omega_2)$, a contradiction.
\end{proof}

\begin{Remark} \label{rem:G2p3}
	As in the previous subsection, Theorems \ref{thm:introA} and \ref{thm:introB} are immediate from Proposition~\ref{prop:G2p3}, see Proposition \ref{prop:F4p2} and Remark \ref{rem:F4p2}.
\end{Remark}

\section{The reduction theorem} \label{sec:reductiontheorem}

In order to prove Theorem \ref{thm:introC}, we will need the following result about indecomposability of twisted tensor products. The proof was suggested to the author by Stephen Donkin.

\begin{Lemma} \label{lem:Donkin}
	Let $\lambda\in X_1$ and let $M$ be an indecomposable rational $G$-module. Then $L(\lambda)\otimes M^{[1]}$ is indecomposable.
\end{Lemma}
\begin{proof}
	As $M^{[1]}$ is a trivial $G_1$-module and $L(\lambda)\vert_{G_1}$ is simple, we have isomorphisms of $G$-modules
	\[ \Hom_{G_1}(L(\lambda),L(\lambda)\otimes M^{[1]}) \cong \Hom_{G_1}(L(\lambda),L(\lambda)) \otimes M^{[1]} \cong M^{[1]} . \]
	Furthermore, $L(\lambda)\otimes M^{[1]}$ is semisimple and $L(\lambda)$-isotypic as a $G_1$-module, so any non-trivial decomposition of $L(\lambda)\otimes M^{[1]}$ as a direct sum of $G$-submodules would afford a non-trivial decomposition of the $G$-module $\Hom_{G_1}(L(\lambda),L(\lambda)\otimes M^{[1]}) \cong M^{[1]}$.
	As $M$ is indecomposable, so is $M^{[1]}$ and it follows that $L(\lambda)\otimes M^{[1]}$ is indecomposable.
\end{proof}

\begin{Theorem} \label{thm:reduction}
	Let $\lambda,\mu\in X^+$ and write $\lambda=\lambda_0+p\lambda_1$ and $\mu=\mu_0+p\mu_1$ with $\lambda_0,\mu_0\in X_1$ and $\lambda_1,\mu_1\in X^+$. Then $L(\lambda)\otimes L(\mu)$ is completely reducible if and only if $L(\lambda_0)\otimes L(\mu_0)$ and $L(\lambda_1)\otimes L(\mu_1)$ are completely reducible.
\end{Theorem}
\begin{proof}
	We have $L(\lambda)=L(\lambda_0)\otimes L(\lambda_1)^{[1]}$ and $L(\mu)=L(\mu_0)\otimes L(\mu_1)^{[1]}$ by Steinberg's tensor product theorem and therefore
	\[ L(\lambda)\otimes L(\mu) = L(\lambda_0)\otimes L(\mu_0)\otimes \big( L(\lambda_1)\otimes L(\mu_1) \big)^{[1]} . \]
	Suppose first that $L(\lambda)\otimes L(\mu)$ is completely reducible. Then $L(\lambda)\otimes L(\mu)$ is completely reducible as a $G_1$-module. As $( L(\lambda_1)\otimes L(\mu_1) )^{[1]}$ is a trivial $G_1$-module, it follows that $L(\lambda_0)\otimes L(\mu_0)$ is completely reducible as a $G_1$-module, hence as a $G$-module by Theorem \ref{thm:introB}.
	
	Now let $M$ be an indecomposable direct summand of $L(\lambda_1)\otimes L(\mu_1)$ and let $L$ be a simple direct summand of $L(\lambda_0)\otimes L(\mu_0)$, so that $L\otimes M^{[1]}$ is a direct summand of $L(\lambda)\otimes L(\mu)$. Then $L$ is $p$-restricted by Theorem \ref{thm:introA} and $L\otimes M^{[1]}$ is indecomposable by Lemma \ref{lem:Donkin}. As $L(\lambda)\otimes L(\mu)$ is completely reducible, $L\otimes M^{[1]}$ is simple and it follows that $M$ is simple. Hence $L(\lambda_1)\otimes L(\mu_1)$ is completely reducible.
	
	Now suppose that $L(\lambda_0)\otimes L(\mu_0)$ and $L(\lambda_1)\otimes L(\mu)$ are completely reducible. By Theorem \ref{thm:introA}, all composition factors of $L(\lambda_0)\otimes L(\mu_0)$ are $p$-restricted, so we have
	\[ L(\lambda_0)\otimes L(\mu_0)\cong L(\delta_1)\oplus\cdots\oplus L(\delta_r) \]
	for certain $p$-restricted weights $\delta_1,\ldots,\delta_r\in X_1$ and
	\[ L(\lambda_1)\otimes L(\mu_1)\cong L(\nu_1)\oplus\cdots\oplus L(\nu_s) \]
	for certain $\nu_1,\ldots,\nu_s\in X^+$. By Steinberg's tensor product theorem, we obtain
	\begin{align*}
	L(\lambda)\otimes L(\mu) & \cong L(\lambda_0)\otimes L(\mu_0)\otimes \big(L(\lambda_1)\otimes L(\mu_1)\big)^{[1]} \\
	& \cong \left( \bigoplus_{i=1}^r L(\delta_i) \right) \otimes \left( \bigoplus_{j=1}^s L(\nu_j) \right)^{[1]} \\
	& \cong \bigoplus_{i=1}^r \bigoplus_{j=1}^s L(\delta_i)\otimes L(\nu_j)^{[1]} \\
	& \cong \bigoplus_{i=1}^r \bigoplus_{j=1}^s L(\delta_i+p\nu_j),
	\end{align*}
	so $L(\lambda)\otimes L(\mu)$ is completely reducible.
\end{proof}

\begin{Corollary}
	Let $\lambda,\mu\in X^+$ and write $\lambda=\lambda_0+\cdots+p^m\lambda_m$ and $\mu=\mu_0+\cdots+p^m\mu_m$ with $\lambda_i,\mu_i\in X_1$ for $i=0,\ldots,m$. Then $L(\lambda)\otimes L(\mu)$ is completely reducible if and only if $L(\lambda_i)\otimes L(\mu_i)$ is completely reducible for $i=0,\ldots,m$.
\end{Corollary}
\begin{proof}
	This follows by induction on $m$ from Theorem \ref{thm:reduction}.
\end{proof}

\bibliographystyle{alpha}
\bibliography{tensor}

\end{document}